\documentclass[11pt,leqno,a4paper]{amsart}
\usepackage{amssymb,amsmath,amscd,graphicx,fontenc,amsthm,mathrsfs, mathtools}
\usepackage[frame, cmtip, arrow, matrix, line, graph, curve]{xy}
\usepackage{graphpap, color}
\usepackage[mathscr]{eucal}
\usepackage{ifthen}
\usepackage{relsize}

\usepackage{fancyhdr}
\usepackage{indentfirst}
\usepackage{paralist}
\usepackage{pstricks}
\usepackage{bbm}
\usepackage{tikz}

\newtheorem{theorem}{Theorem}[section]
\newtheorem{definition}[theorem]{Definition}
\newtheorem{lemma}[theorem]{Lemma}
\newtheorem{proposition}[theorem]{Proposition}

\theoremstyle{remark}
\newtheorem{remark}[theorem]{Remark}

\numberwithin{equation}{section}

\newcommand{\NN}{{\mathbb{N}}}
\newcommand{\ZZ}{{\mathbb{Z}}}

\newcommand{\RR}{{\mathbb{R}}}
\newcommand{\PP}{{\mathbb{P}}}
\newcommand{\QQ}{{\mathbb{Q}}}

\renewcommand{\PP}{{\mathbb{P}}}

\DeclareMathOperator{\vol}{vol}

\newcommand{\diag}{\mathrm{diag}}

\newcommand{\tp}{^\mathrm{t}}

\newcommand{\GL}{\mathrm{GL}}
\newcommand{\SL}{\mathrm{SL}}

\newcommand{\UL}{\mathrm{UL}}

\newcommand{\SO}{{\mathrm{SO}}}

\newcommand{\AUL}{\mathrm{AUL}}

\newcommand{\Id}{\mathrm{Id}}

\newcommand{\SH}{\mathsf{H}}

\newcommand{\sg}{\mathsf{g}}
\newcommand{\sh}{\mathsf{h}}
\newcommand{\sa}{\mathsf{a}}

\newcommand{\T}{{\mathsf{T}}}
\newcommand{\J}{\mathsf{J}}

\newcommand{\I}{\mathsf{I}}

\newcommand{\q}[2]{\mathsf q^{#1}_{#2}}

\newcommand{\vx}{\mathbf x}
\newcommand{\vy}{\mathbf y}
\newcommand{\vw}{\mathbf w}
\newcommand{\vv}{\mathbf v}

\providecommand{\ve}{\mathbf{ e}}
\providecommand{\vp}{\mathbf{p}}

\providecommand{\vk}{\mathbf{k}}
\providecommand{\vm}{\mathbf{m}}

\newcommand{\NT}{\mathbf N}

\newcommand{\Prim}[1]{\mathrm P({#1})}

\renewcommand{\varpi}{\pi}

\newcommand{\pp}{\mathtt p}
\newcommand{\origin}{O}

\newcommand{\ind}{\chi^{}}

\definecolor{cmd}{rgb}{1.0, 0.35, 0.21}

\begin{document}
\title[values of inhomogeneous forms]{Values of inhomogeneous forms at $S$-integral points}
\author{Anish Ghosh}

\address{\textbf{Anish Ghosh} \\
    School of Mathematics,
    Tata Institute of Fundamental Research, Mumbai, India 400005}
\email{ghosh@math.tifr.res.in}

\author{Jiyoung Han}

\address{\textbf{Jiyoung Han} \\
    School of Mathematics \\
Tata Institute of Fundamental Research, Mumbai, India 400005}
\email{jyhan@math.tifr.res.in}

\thanks{A.\ G.\ gratefully acknowledges support from a MATRICS grant from the Science and Engineering Research Board, a grant from the Infosys foundation and a Department of Science and Technology, Government of India, Swarnajayanti fellowship}

\maketitle
\begin{abstract}
We prove effective versions of Oppenheim's conjecture for generic inhomogeneous forms in the $S$-arithmetic setting. We prove an effective result for fixed rational shifts and generic forms and we also prove a result where both the quadratic form and the shift are allowed to vary. In order to do so, we prove analogues of Rogers' moment formulae for $S$-arithmetic congruence quotients as well as for the space of affine lattices. We believe the latter results to be of independent interest. 
\end{abstract}

\tableofcontents

 
\section{Introduction}
Let $q$ be a quadratic form and consider the \emph{inhomogeneous} quadratic form $q_{\xi}(\vv) := q(\vv + \xi)$ for $\vv \in \RR^n$. Here, we refer to $\xi \in \RR^n$ as a shift. We say that $q_{\xi}$ is \emph{indefinite} if $q$ is indefinite and \emph{non-degenerate} if $q$ is non-degenerate. Moreover $q_{\xi}$ is said to be \emph{irrational} if either $q$ is an irrational quadratic form, i.e. not proportional to a quadratic form with integer coefficients, or $\xi$ is an irrational vector. The purpose of this paper is to study values taken at integer points by inhomogeneous quadratic forms at $S$-integer points. 

A celebrated theorem of Margulis \cite{Mar} resolving an old conjecture of Oppenheim, states that for an indefinite irrational quadratic form $q$ in $n \geq 3$ variables, $q(\ZZ^n)$ is dense in $\RR$. Further fundamental work in this direction was carried out by Dani and Margulis \cite{DM} and by Eskin, Margulis and Mozes \cite{EMM, EMM2} who proved quantitative versions of Oppenheim's conjecture under suitable hypotheses.

Inhomogeneous forms arise in a variety of situations in mathematics and physics and were studied by Marklof in his fundamental work \cite{Mar2003, Mar2002} on pair correlation densities and their relation to the Berry Tabor conjecture.  In \cite{MaMo}, Margulis and Mohammadi proved a quantitative version of Oppenheim's conjecture for inhomogeneous forms. They showed that for any indefinite, irrational and non-degenerate inhomogeneous form $q_{\xi}$ in $n \geq 3$ variables there is $c_q>0$ such that 
$$\liminf_{t\to \infty}\frac{\NT(q_{\xi},I, t)}{t^{n-2}}\geq c_q |I|.$$
For $n\geq 5$ they showed that the above limit exists and equals $c_Q|I|$. Here $\NT(q_{\xi},I, t)$ is the counting function 
\begin{equation}\label{equ:countingfunction}
\NT(q_{\xi},I, t)=\#\{\vv\in\ZZ^n\ |\  q_{\xi}(\vv)\in I,\ \|\vv\|\leq t\},
\end{equation}
where $I\subseteq \RR$ is an interval and $\|\cdot\|$ is the Euclidean norm on $\RR^n$. 
 This result  is an inhomogeneous analogue of the results of Dani-Margulis and Eskin-Margulis-Mozes alluded to above. It constitutes a quantitative strengthening of, and therefore implies, an inhomogeneous analogue of the Oppenheim conjecture: namely that $q_{\xi}(\mathbb{Z}^{n})$ is dense in $\mathbb{R}$ for any indefinite, irrational, non-degenerate inhomogeneous form $q_{\xi}$ in $n \geq 3$ variables. We refer the reader to \cite{BG} for a self-contained proof of this (qualitative) density. In \cite{GKY2020, GKY2020b}, Ghosh, Kelmer and Yu considered an \emph{effective} version of the above inhomogeneous theorem for \emph{generic} forms. The term ``effective" is used in the literature to refer to two related problems: the issue of error terms in the counting problem above, as well as the question of obtaining a good bound for $\vv$ which ensures that $0 < q_{\xi}(\vv) < \varepsilon$. The term ``generic" refers to a full measure set of forms for an appropriate measure.
 
  As far as inhomogeneous forms are concerned, there are three natural regimes in which one could study this problem. To wit, one could allow both the form and the shift to vary, i.e. make the problem generic in two variables. Or alternatively, one could fix either the form or the shift and allow the other to vary. 
 In Theorem $1.1$ \cite{GKY2020}, Ghosh Kelmer and Yu noted that an effective result where both form and shift are allowed to vary, follows from an affine analogue of Rogers second moment formula for the space of affine lattices in conjunction with methods from an earlier paper \cite{KY2018} of Kelmer and Yu (which is discussed below). The case where either the form or the shift is fixed is significantly more difficult. One of the main results in \cite{GKY2020} is an effective theorem for fixed shifts and generic quadratic forms. In order to prove such a theorem, the authors prove an analogue of Rogers' second moment formula on congruence quotients of $\SL_{n}(\RR)$. The complementary case of  fixed forms and generic shifts is studied in \cite{GKY2020b} using a very different technique, namely effective mean ergodic theorems.

Our aim in this paper is to prove an $S$-arithmetic effective Oppenheim theorem for inhomogeneous forms. In Theorem \ref{inhomogeneous case} we allow both form and shift to vary, while in Theorem \ref{congruence case}, we treat the more difficult problem of fixing a rational shift and allowing the form to vary. We broadly follow the strategy of \cite{GKY2020} where the main tool is an analogue of Rogers' second moment formula for Siegel transforms on congruence quotients of $\SL_{n}(\RR)$. Accordingly, we prove two new analogues of Rogers' formula. In Theorem \ref{moment formulae: inhomogeneous} we establish an $S$-arithmetic analogue of Rogers' formula for the space of unimodular affine lattices. In Theorem \ref{moment theorems} which is the technical heart of the paper, we obtain an analogue of Rogers' formula for  congruence quotients. We believe these results, especially Theorem \ref{moment theorems} to be of independent interest. Indeed, the moment formula established in \cite{GKY2020} has already found other Diophantine applications, cf. \cite{AGY20}. Rogers' type results have found wide applicability in recent years.  

Effective results for Oppenheim type problems have recently received considerable attention. In \cite{LM14}, Lindenstrauss and Margulis proved a remarkable effective result for ternary quadratic forms. Their result is valid for all irrational quadratic forms satisfying an explicit Diophantine condition. If one settles for a generic set of forms instead, then better results can be obtained. In \cite{GGN}, Ghosh, Gorodnik and Nevo used effective mean ergodic theorems to prove a variety of results of this flavour including the case of ternary quadratic forms, namely the classical Oppenheim conjecture. The idea of using Rogers' mean value formula to study effective versions of Oppenheim's conjecture is due to Athreya and Margulis \cite{AM} and was further developed by Kelmer and Yu \cite{KY2018}. We also mention the work of Bourgain \cite{Bour} on certain `uniform' versions of Oppenheim's conjecture for  diagonal forms. See also \cite{GK} for an analogue for ternary forms and \cite{BGH, KS19}. 

The study of values of quadratic forms at integer points in the $S$-arithmetic setting was initiated in the work of Borel and Prasad \cite{BP} who proved analogues of Margulis's theorem. In \cite{HLM2017}, Han, Lim and Mallahi-Karai obtained $S$-arithmetic generalizations of the quantitative Oppenheim problem in rank $5$ and higher. The case of forms of rank $3$ and $4$ was studied by Han in \cite{Han2020}. Effective results for generic (homogeneous) forms in the $S$-arithmetic setting were obtained by Han in \cite{Han2021}, who also established an $S$-arithmetic version of Rogers' mean value formula. Finally, we mention that the question of values of \emph{rational} quadratic forms with congruence conditions has also received attention recently, see for instance \cite{BG}.

\section{Notation and Results}\label{statements}
\subsection{The Set-up}
Let $S$ be a finite set of places over $\QQ$ including the infinite place, with the usual corresponding Euclidean norm $\|\cdot\|_\infty$. Let us denote by $S_f=\{p_1, \ldots, p_s\}=S-\{\infty\}$, where each prime $p_j$ represents the $p_j$-adic norm on $\QQ$, the set of finite places in $S$. Define $\QQ_S=\prod_{p\in S} \QQ_p$, where $\QQ_p$ is the completion of $\QQ$ with respect to the norm $\|\cdot\|_p$. Here, $\QQ_p=\RR$ when $p=\infty$. Define \emph{the ring of $S$-integers} by
\[
\ZZ_S=\left\{(z,\ldots, z)\in \QQ_S : z\in \QQ \text{ with } |z|_v\le 1 \text{ for }\forall v\notin S_f\right\}.
\]

For $p=\infty$, let $\vol_\infty$ be the canonical Lebesgue measure on $\RR^d$, $d\ge 1$, and for each $p\in S_f$, we will consider the Haar measure $\vol_p$ on $\QQ_p^d$ normalized by $\vol_p(\ZZ_p^d)=1$. Let us define the measure on $\QQ_S^d$ as $\vol=\prod_{p\in S} \vol_p$. We also use the notation $\vol^{(d)}$ when we want to specify the dimension of the space. 

We say that $\Lambda \subseteq \QQ_S^d$ is an \emph{$S$-lattice} if $\Lambda$ is a free $\ZZ_S$-module of rank $d$ and an \emph{affine $S$-lattice} if it is of the form $\Lambda'+\xi$ for some $S$-lattice $\Lambda'$ and an element $\xi\in \QQ_S^d$. 
An $S$-lattice (affine $S$-lattice, respectively) is \emph{unimodular} if $\vol(\Lambda\setminus\QQ_S^d)=1$ ($\vol(\Lambda'\setminus \QQ_S^d)=1$, respectively).

For $p\in S_f$, define
\[
\UL_d(\QQ_p):=\left\{ g_p\in \GL_d(\QQ_p) : |\det g_p|_p=1\right\}
\]
and let $\UL_d(\RR):=\SL_d(\RR)$. Denote $\UL_d(\QQ_S)=\prod_{p\in S} \UL_d(\QQ_p)$ and let $\sg=(g_p)_{p\in S}$ be an element of $\UL_d(\QQ_S)$.
Consider the group 
$$\AUL_d(\QQ_S)=\{(\xi, \sg): \xi\in \QQ_S^d,\; \sg\in \UL_d(\QQ_S)\}$$ 
with the binary operation
\[
(\xi_1, \sg_1)(\xi_2, \sg_2)=(\xi_2+\sg_2\xi_1, \sg_1\sg_2).
\]

One can identify the space of unimodular affine $S$-lattices in $\QQ_S^d$ with ${\AUL_d(\ZZ_S) \setminus \AUL_d(\QQ_S)}$ using the correspondence 
\[\AUL_d(\ZZ_S)(\xi, \sg)\;\leftrightarrow\; \ZZ_S^d\sg+\xi.
\]


\subsection{Notational Remarks}
\begin{enumerate}
\item We will use sans serif typestyle font for parameters of an $S$-arithmetic space as we already use $\sg=(g_p)_{p\in S}$ for an element of an $S$-arithmetic Lie group. For $\T=(T_p)_{p\in S}$, an element of $\RR_{>0}\times \prod_{p\in S_f} p^{\ZZ}\;\left(\subseteq \RR_{>0}^{s+1}\right),$ we say that $\T=(T_p)_{p\in S} \succeq \T'=(T'_p)_{p\in S}$ when $T_p\ge T'_p$ for each $p\in S$, and $\T\rightarrow \infty$ if $T_p\rightarrow \infty$ for all $p\in S$.
\item For $S=\{\infty, p_1, \ldots, p_s\}$, define
\[\begin{split}
\NN_S&=\left\{q\in \NN : \gcd(q, p_1\cdots p_s)=1\right\}\;\text{and}\\
\PP_S&=\left\{p_1^{z_1}\cdots p_s^{z_s}\in \NN : z_j\in \NN\cup\{0\}\right\}.
\end{split}\]

We remark that there is a one-to-one correspondence between $\NN_S\times \PP_S$ and $\NN$.
If we denote by $\Prim{\ZZ_S^d}(=\ve_1.\SL_d(\ZZ_S))$ ($\Prim{\ZZ^d}$, respectively) the set of primitive elements of $\ZZ_S^d$ ($\ZZ^d$, respectively), then one can easily check that $\Prim{\ZZ_S^d}=\PP\cdot\Prim{\ZZ^d}$. 
\end{enumerate}

A quadratic form $\q{}{}$ on $\QQ_S^d$ is a collection $\q{}{}=(q^{(p)})_{p\in S}$ of quadratic forms $q_p$ over $\QQ_p$ for each $p\in S$.
We say that $\q{}{}$ is \emph{non-degenerate} (\emph{isotropic}, respectively) if $q^{(p)}$ is non-degenerate (isotropic, respectively) for all $p\in S$. Recall that a form is \emph{isotropic} if there exists a non-zero vector $\vv \in \QQ^{n}_p$ such that $q_p(\vv) = 0$ and that for the infinite place, this condition is equivalent to being indefinite.

We say that the collection $\{\I_\T : \T \}$ of Borel subsets of $\QQ_S$ is \emph{decreasing} if $\I_\T' \subseteq \I_\T$ whenever $\T \succeq \T'$. Our first main theorem is an effective counting statement for quadratic forms with congruence conditions in the $S$-arithmetic setting.

\begin{theorem}\label{congruence case}
Let $d\ge 3$. Let $\mathcal I=\{\I_\T=(I^{(p)}_{T_p})_{p\in S} : \T \}$ be a decreasing family of bounded Borel sets in $\QQ_S$ satisfying
\[
\vol_\infty(I^{(\infty)}_{T_\infty})=c_\infty T_\infty^{-\kappa_\infty}
\quad\text{and}\quad
I^{(p)}_{T_p}=a_p+ p^{c_p+\kappa_pt_p}\ZZ_p\;(p\in S_f)
\]
for some $(c_p)_{p\in S_f}\in \RR_{>0}\times \prod_{p\in S_f} \ZZ$, $a_p\in \QQ_p$ ($p\in S_f$) and
\begin{equation}\label{restriction from vol formula}
0\le \kappa_\infty<d-2
\quad\text{and}\quad
\left\{\begin{array}{lc}
\kappa_p\in \{0,1\},& \text{if } d\ge 4;\\
\kappa_p=0,& \text{if } d=3.\end{array}\right.
\end{equation}
Let $q\in \NN_S$ and $\vp \in \ZZ_S^d$. For a quadratic form $\q{}{}$ on $\QQ_S^d$ and $\T$, define
\[
\NT(q, \vp\: ;\q{}{}, \mathcal I, \T)
=\#\left\{\vk\in q\ZZ_S^d+\vp :  
\|\vk\|_p < T_p\;(p\in S)\;\text{and}\; \q{}{}(\vk)\in \I_\T\right\}.
\]

There is $\delta_0>0$ such that for any $\delta\in (0, \delta_0)$, we have
\[
\NT(q, \vp\: ;\q{}{}, \mathcal I,\T)
=c_{\q{}{}}\frac{1}{q^d}\vol(\I_\T) |\T|^{d-2}
+o(\vol(\I_\T)|\T|^{d-2-\delta})
\]
for almost all non-degenerate isotropic quadratic forms $\q{}{}$, where $|\T|=\prod_{p\in S} T_p$.
Here, the implied constant of the error term is uniform on a compact set of the space of non-degenerate isotropic quadratic forms, and depends on $q$.
\end{theorem}

\begin{remark}
We remark that the reason that the condition for the finite place in \eqref{restriction from vol formula} is more restrictive than that of the infinite place is because of the difference in method used for the volume estimation for given regions in $\RR^d$ and $\QQ_p^d$ (see Theorem~\ref{volume formula}).
\end{remark}

\begin{remark}
We note that if $q=1$, then Theorem~\ref{congruence case} follows from \cite{Han2021}. Hence throughout this paper, let us assume that $1\neq q\in \NN_S$.
\end{remark}

Observe that putting $\vk_1=\frac 1 q \vk$, we have
\begin{equation}\label{relation}\begin{split}
&\NT(q,\vp;\q{}{}, \mathcal I, \T)\\
&\hspace{0.2in}=\#\left\{\vk_1\in \ZZ_S^d+\frac {\vp} q : \begin{array}{c}
\|\vk_1\|\infty< \frac {T_\infty} q; \\
\|\vk_1\|_p < T_p\;(p\in S_f)\end{array}\;\text{and}\;
\q{}{}(\vk_1) \in \frac 1 {q^2} \I_\T\;\right\}.
\end{split}\end{equation}

 As observed in \cite{GKY2020}, values of inhomogeneous forms with rational shifts at integer points are simply scaled values of the homogeneous form at integer points satisfying a congruence condition. 
 
 Our next result is an effective counting result for inhomogeneous forms in the $S$-arithmetic setting where both the form and the shift are allowed to vary.
\begin{theorem}\label{inhomogeneous case}
Let $d\ge 3$ and let $\mathcal I=\{\I_\T\}$ be as Theorem~\ref{congruence case}.
For a quadratic form $\q{}{}$ and $\xi\in \QQ_S^d$, define 
\[
\NT(\q{}{\xi}, \mathcal I, \T)= \#\left\{
\vk \in \ZZ_S^d : \|\vk\|_p < T_p\;(p\in S) \;\text{and}\; \q{}{\xi}(\vk)\in \I_\T\right\}.
\]

There is $\delta_0>0$ such that for any $\delta\in (0,\delta_0)$,
\[
\NT(\q{}{\xi},\mathcal I, \T)=c_{\q{}{}}\vol(\I_\T) |\T|^{d-2} + o_{\q{}{\xi}}(\vol(\I_\T)|\T|^{d-2-\delta})
\]
for almost all pairs $(\q{}{}, \xi)$ of a non-degenerate isotropic quadratic form and an element of $\QQ_S^d$. 
\end{theorem}

\section{Volume Computation for $\UL_d(\ZZ_S)\setminus\UL_d(\QQ_S)$ and $\SL_d(\ZZ_S)\setminus\SL_d(\QQ_S)$}\label{Subsect: volume}

We first construct Haar measures $\mu^{(d)}_p$ and $\nu^{(d)}_p$ on $\UL_d(\QQ_p)$ and $\SL_d(\QQ_p)$, respectively so that Haar measures $\mu^{(d)}$ and $\nu^{(d)}$ of $\UL_d(\QQ_S)$ and $\SL_d(\QQ_S)$ can be defined as $\mu^{(d)}=\prod_{p\in S} \mu^{(d)}_p$ and $\nu^{(d)}=\prod_{p\in S} \nu^{(d)}_p$, respectively.

Let $\SH$ be the subgroup of $\UL_d(\QQ_S)$ given by
\begin{equation}\label{dual AUL}
\SH=\left\{\left(\begin{array}{c|cccc}
1 & 0 & 0 & \cdots & 0 \\
\hline
 & & & & \\
\tp\vv' & & \sg' & & \\
& & & & \\
& & & & \\
\end{array}\right): \begin{array}{c}
\vv' \in \QQ_S^{d-1},\\
\sg'\in \UL_{d-1}(\QQ_S)\end{array}\right\}.
\end{equation}

Set $\mu^{(1)}_p(\UL_1(\QQ_p))=\mu^{(1)}(\ZZ_p - p\ZZ_p)=1-1/p$ and $\nu^{(1)}(\SL_1(\QQ_p))=\nu^{(1)}(\{1\})=1$. 
For $d\ge 2$, consider the map 
\[\Phi_p : \left\{\begin{array}{c}
\SH \\
\SH \cap \SL_d(\QQ_p) \end{array}\right. \times \left\{\vy=(y_1, \ldots, y_d)\in \QQ_p^d : y_1 \neq 0 \right\} \rightarrow \left\{\begin{array}{c}
\UL_d(\QQ_p)\\
\SL_d(\QQ_p)\end{array}\right.\]
given by
\begin{equation}\label{Map Phi}
\begin{split}
&\Phi_p : \left((\vv', \sg'), \vy \right) \mapsto \\
&\hspace{0.5in}A=\left(\begin{array}{c|cccc}
1 & 0 & 0 & \cdots & 0 \\
\hline
 & & & & \\
\tp\vv' & & \sg' & & \\
& & & & \\
& & & & \\
\end{array}\right)
\left(\begin{array}{c|cccc}
y_1 & y_2 & y_3 & \cdots & y_d \\
\hline
0 & 1/y_1 & & & \\
0 & & 1 & & \\
\vdots & & & \ddots & \\
0 & & & & 1 \\ 
\end{array}
\right).
\end{split}\end{equation}

Note that $\Phi_p$ is a diffeomorphism and its image covers $\UL_d(\QQ_p)$ ($\SL_d(\QQ_p)$, respectively) outside a co-dimension one submanifold.
Let $\mu^{(d-1)}_p$ and $\nu^{(d-1)}_p$ be Haar measures on $\UL_{d-1}(\QQ_p)$ and $\SL_{d-1}(\QQ_p)$, respectively, from the induction hypothesis. 
It is well-known that $\vol^{(d-1)}\cdot\mu_p^{(d-1)}$ and $\vol^{(d-1)}\cdot\nu_p^{(d-1)}$, where $\vol^{(d-1)}$ is the measure on $\QQ_p^{d-1}$ defined in Section~\ref{statements}, are Haar measures on $\SH$ and $\SH\cap \SL_d(\QQ_p)$, respectively. 

\begin{proposition}\label{Def Vol}
The push-forward measures
\[\begin{split}
\mu_p^{(d)}(\sg)&:=(\Phi_p)_* d\vv' d\mu_p^{(d-1)}(\sg') d\vy\;\text{and}\\
\nu_p^{(d)}(\sg)&:=(\Phi_p)_* d\vv' d\nu_p^{(d-1)}(\sg') d\vy
\end{split}\] 
give Haar measures on $\UL_d(\QQ_p)$ and $\SL_d(\QQ_p)$, respectively. Here, we simply denote $d\vv'=d\vol^{(d-1)}(\vv')$ and $d\vy=d\vol^{(d)}(\vy)$.
\end{proposition}
\begin{proof} Since both groups admit lattice subgroups, Haar measures on them are unimodular. Hence it suffices to show that $$(\Phi_p)_*d\vv' d\mu_p^{(d-1)}(\sg')d\vy$$ 
is $\UL_d(\QQ_p)$-invariant under right multiplication $R_g$. We omit the argument for $\SL_d(\QQ_p)$ since the argument is identical. 

Since the image of $\Phi_p$ covers $\UL_d(\QQ_p)$ outside a co-dimension one submanifold, we may choose a generic element $\sh\in \UL_d(\QQ_p)$ of the form $\sh=\Phi_p\left((\vw', \sh'), \vx\right)$. Let $\vx=(x_1, \ldots, x_d)$ and $\vy=(y_1, \ldots, y_d)$. Then
\[\begin{split}
&\left((\vv', \sg'), \vy\right)\left((\vw', \sh'), \vx\right)\\
&=\left((\vv'+\vw'_1 \sg', \sg'\sh'_1), (z_1x_1, z_1(x_2, \ldots, x_d)+ (z_2, \ldots, z_d)\diag(x_1^{-1}, 1, \ldots, 1)\right)\\
&=:\left((\vv'_1, \sg'_1), \vy_1\right),
\end{split}\]
where $z_1={y_1+(y_2, \ldots, y_d)\tp\vw'}$, $\vw'_1 =\frac 1 {z_1}\:\vw'\:\diag(y_1^{-1}, 1, \ldots, 1),$ and
$$\sh'_1 =\left(\diag(y_1^{-1},1,\ldots,1)\left(\Id_{d-1}- \frac 1 {z_1} {\tp \vw'} (y_2, \ldots, y_d)\right)\sh'\right)\diag(z_1,1,\ldots,1).$$
Here, $\Id_{d-1}$ is the identity matrix of size $d-1$.

Since the Jacobi matrices $d\vy_1/d\sg'$, $d\vy_1/d\vv'$ and $d\mu_p^{(d-1)}(\sg'_1)/d\vv'$ are trivial and $d\vv'_1/d\vv'=\Id_{d-1}$, we have
\[
(R_g)_* d\vv' d\mu^{(d-1)}(\sg') d\vy 
= \left|\frac {d\mu_p^{(d-1)}(\sg'_1)}{d\mu_p^{(d-1)}(\sg')}\right|\left| \frac {d\vy_1}{d\vy}\right| d\vv' d\mu^{(d-1)}(\sg') d\vy.
\]
Since $\sg'_1=\sg'\sh'_1$ and $\sg'$ is irrelevant to $\sh'_1$, by the induction hypothesis that $\mu_p^{(d-1)}$ is $\UL_{d-1}(\QQ_p)$-invariant, we have that $$\left|d\mu^{(d-1)}_p(\sg'_1)/d\mu^{(d-1)}_p(\sg')\right|=1.$$
One can also show that $\left|d\vy_1/d\vy\right|=1$ by direct computation.
\end{proof}

When $p=\infty$, the measure $\mu_{\infty}^{(d)}$ of $\SL_d(\RR)$ can be identified with
\[
d\mu_{\infty}^{(d)}=\delta(1-\det g_\infty)\prod_{1\le i,j \le d} dg_{ij},
\]
where $g_\infty=(g_{ij})\in \SL_d(\RR)$ and $\delta$ is the Dirac-delta distribution (see formulas (3.67) and (3.70) in \cite{Mar2000}).

One can find volume formulas of the quotients of semisimple $S$-arithmetic Lie groups, based on the Iwasawa ($KAN$) decomposition of semisimple Lie groups in the work of Prasad \cite{Pra1989}. 
However, this method does not fit for our usage in the next section as well as $\UL_d(\QQ_S)$ is not semisimple. 
Hence let us devote some space to compute the volumes of the quotient spaces $\UL_d(\ZZ_S)\setminus \UL_d(\QQ_S)$ and $\SL_d(\ZZ_S)\setminus \SL_d(\QQ_S)$, adjusting the method provided by Siegel \cite{Sie1989}. 

For a positive integer $d\ge 2$, define
\[
\zeta^{}_{S}(d):=\sum_{t \in \NN_S} \frac 1 {t^d}.
\]

For a bounded and compactly supported function $f:\QQ_S^d\rightarrow \RR_{\ge 0}$, we define the \emph{homogenous Siegel transform} $\widetilde f$ by
\[
\widetilde{f}(\Lambda)=\sum_{\vv\in \Lambda-\{\origin\}} f(\vv), \; \forall \;\Lambda \in \UL_d(\ZZ_S)\setminus\UL_d(\QQ_S).
\]

Define $\alpha : \UL_d(\ZZ_S)\setminus \UL_d(\QQ_S) \rightarrow \RR_{\ge 0}$ by
\[
\alpha(\Lambda)=\sup_{1\le j \le d}
\left\{ \frac 1 {\|\vv_1 \wedge \cdots \wedge \vv_j\|} :
\begin{array}{c}
\vv_1, \ldots, \vv_j \in \Lambda\\
\text{linearly independent} \end{array}
\right\}.
\]

The transform $\widetilde f$ has the following property (\cite[Proposition 3.3]{Han2021}, see also \cite{Sch1971}).
\begin{theorem}\label{HLM Lemma 3.8} Let $f: \QQ_S^d \rightarrow \RR_{\ge 0}$ be bounded and compactly supported. Then there is $c=c(f)>0$ for which 
$$\widetilde f(\Lambda)< c\alpha(\Lambda)$$ 
for any $\Lambda \in \UL_d(\ZZ_S)\setminus \UL_d(\QQ_S)$.
The constant $c>0$ can be taken uniformly among the family of dilates of $f$ with appropriate normalization.
\end{theorem}

The following theorem is found in \cite{HLM2017} (Lemma 3.10) for $\SL_d(\ZZ_S)\setminus \SL_d(\QQ_S)$ and \cite{Han2021} (Proposition 3.2) for $\UL_d(\ZZ_S)\setminus \UL_d(\QQ_S)$.

\begin{theorem}\label{HLM Lemma 3.10} For $1\le r < d$, 
\[
\int_{\UL_d(\ZZ_S)\setminus \UL_d(\QQ_S)} \alpha^r d\mu_S^{(d)} < \infty
\quad\text{and}\quad
\int_{\SL_d(\ZZ_S)\setminus \SL_d(\QQ_S)} \alpha^r d\mu_S^{(d)} < \infty.
\]
\end{theorem}

We now come to the main volume estimate of this paper.
\begin{theorem}\label{Com Vol}
Let $\mu^{(d)}_S=\prod_{p\in S} \mu_p^{(d)}$ and $\nu^{(d)}_S=\prod_{p\in S} \nu_p^{(d)}$ be Haar measures on $\UL_d(\QQ_S)$ and $\SL_d(\QQ_S)$, respectively, where $\mu_p^{(d)}$ and $\nu_p^{(d)}$ are defined in Proposition \ref{Def Vol}. Then
\[\begin{split}
\mu^{(d)}_S\left(\UL_d(\ZZ_S)\setminus \UL_d(\QQ_S)\right)
&=\prod_{p\in S_f} \left(1-\frac 1 p\right)\cdot\zeta^{}_S(d)\zeta^{}_S(d-1)\cdots\zeta^{}_S(2)\;\text{and}\\
\nu^{(d)}_S\left(\SL_d(\ZZ_S)\setminus \SL_d(\QQ_S)\right)
&=\zeta^{}_S(d)\zeta^{}_S(d-1)\cdots\zeta^{}_S(2).
\end{split}\]
\end{theorem}
\begin{proof}
As in Proposition \ref{Def Vol}, it suffices to show the statement for  $\UL_d(\QQ_S)$ since the proof of the $\SL_d(\QQ_S)$-case will be almost identical except the initial value $\nu_S^{(1)}\left(\SL_1(\ZZ_S)\setminus \SL_1(\QQ_S)\right)=1$, while $\mu_S^{(1)}\left(\UL_1(\ZZ_S)\setminus \UL_1(\QQ_S)\right)={\prod_{p\in S_f}(1-1/p)}$.
For simplicity, let us denote 
$$V_d=\mu_S^{(d)}\left(\UL_d(\ZZ_S)\setminus \UL_d(\QQ_S)\right).$$

Fix a fundamental domain $\mathcal F$ for $\UL_d(\ZZ_S)\setminus \UL_d(\QQ_S)$. Let $f:\QQ_S^d\rightarrow \RR$ be any given bounded continuous function of compact support.
Consider a sequence $\left(\lambda_n=(e^{-n}, p_1^{n}, \ldots, p_s^{n})\right)_{n\in \NN}$ of $S$-arithemetic numbers. It is obvious that for any $\sg\in \UL_d(\QQ_S)$,
\[\begin{split}
I&:= \int_{\QQ_S^d} f(\vv)d\vv
=\lim_{n\rightarrow\infty} d(\lambda_n)^{d}
\sum_{\vv\in \ZZ_S^d.\sg-\{0\}} f(\lambda_n\vv)\\
&=\lim_{n\rightarrow\infty} d(\lambda_n)^{d}\: \widetilde {f_{\lambda_n}}(\sg),
\end{split}\]
where $f_{\lambda_n}(\vv)=f(\lambda_n\vv)$ and $d(\lambda_n)=(ep_1\cdots p_s)^{-n}$ (Here, $d(\cdot)$ stands for the covolume of the lattice $\ZZ_S.\lambda_n \subseteq \QQ_S$).

By Lebesgue's dominated convergence theorem, using Theorem~\ref{HLM Lemma 3.8} and Theorem~\ref{HLM Lemma 3.10}, it follows that
\[
V_d I=\int_{\mathcal F} \lim_{n\rightarrow \infty}
d(\lambda_n)^{d}\: \widetilde {f_{\lambda_n}}(\sg) d\mu_S^{(d)}(\sg)
= \lim_{n\rightarrow \infty} d(\lambda_n)^{d} \int_{\mathcal F} \widetilde {f_{\lambda_n}}(\sg) d\mu_S^{(d)}(\sg).
\]

Put
\[\begin{split}
\xi&=d(\lambda_n)^{d} \int_{\mathcal F} \widetilde {f_{\lambda_n}}(\sg) d\mu_S^{(d)}(\sg)\;\text{and}\\
\xi_t&=d(\lambda_n)^{d} \sum_{\vv \in t\Prim{\ZZ_S^d}.\sg} \int_{\mathcal F}f(\lambda\vv)d\mu_S^{(d)}(\sg), \; \forall\; t \in \NN_S
\end{split}\]
so that $\xi=\sum_{t\in \NN_S} \xi_t$. We will compute $\xi_t$ using the induction hypothesis of dimension $d-1$ and hence obtain $\xi$. During the proof, we will see that $\xi$ and $\xi_t$ are eventually not relavant to $\lambda_n$ (which is why we do not use the notation $\xi_{n}$ or $\xi_{n,t}$ instead).

For each $\vk\in \Prim{\ZZ_S^d}$, assign an element $\sg_{\vk}\in \SL_d(\ZZ_S)$ for which $\vk=\ve_1.\sg_{\vk}$, where $\{\ve_i : 1\le i \le d\}$ is the canonical basis of $\QQ_S^d$. Denote by $[\sg]_1$ the first row of the matrix $\sg\in \SL_d(\QQ_S)$. Then
\[
d(\lambda_n)^{-d}\xi_1
=\sum_{\vk\in \Prim{\ZZ_S^d}} \int_{\sg_{\vk}\mathcal F} f(\lambda_n [\sg]_1) d\mu_S^{(d)}(\sg)
=\int_{\mathcal F_1} f(\lambda_n \sg)d\mu_S^{(d)}(\sg),
\]
where $\mathcal F_1:=\bigsqcup_{\vk\in \Prim{\ZZ_S^d}} \sg_{\vk}\mathcal F$.

Since $\big\{\sg_{\vk} : \vk\in \Prim{\ZZ_S^d}\big\}$ is the set of representatives of $\SH_{\ZZ_S}\setminus \UL_d(\ZZ_S)$, 
where $\SH_{\ZZ_S}:=\SH\cap \UL_d(\ZZ_S)$,
\[
\UL_d(\QQ_S)
=\bigsqcup_{\gamma\in\UL_d(\ZZ_S)}\mathcal F
=\bigsqcup_{\gamma'\in\SH_{\ZZ_S}}\bigsqcup_{\vv\in \Prim{\ZZ_S^d}}\sg_v \mathcal F
=\bigsqcup_{\gamma'\in\SH_{\ZZ_S}} \mathcal F_1
\]
so that $\mathcal F_1$ is the fundamental domain of $\SH_{\ZZ_S}$ in $\UL_d(\QQ_S)$.

On the other hand, by considering the inverse map of $\Phi=\prod_{p\in S} \Phi_p$, where $\Phi_p$ is given as in \eqref{Map Phi}, any element $\sg=\big(g_p=(g^{(p)}_{ij})\big)_{p\in S} \in \UL_d(\QQ_S)$ such that $g^{(p)}_{11}\neq 0$, for all $p\in S$, is uniquely expressed as 
\[\sg= \left(\begin{array}{c|cccc}
1 & 0 & 0 & \cdots & 0 \\
\hline
 & & & & \\
\tp\vv' & & \sg' & & \\
& & & & \\
& & & & \\
\end{array}\right)
\left(\begin{array}{c|cccc}
\sg_{11} & \sg_{12} & \sg_{13} & \cdots & \sg_{1d} \\
\hline
0 & 1/\sg_{11} & & & \\
0 & & 1 & & \\
\vdots & & & \ddots & \\
0 & & & & 1 \\ 
\end{array}
\right),
\]
where $\sg'\in \UL_{d-1}(\QQ_S)$ and $\vv' \in \QQ_S^{d-1}$.
Let $\mathcal F'$ be any fundamental domain of $\UL_{d-1}(\ZZ_S)\setminus \UL_{d-1}(\QQ_S)$. Then $(\ZZ_S\setminus \QQ_S)^{d-1}\times \mathcal F'$ is the fundamental domain of $\SH_{\ZZ_S}\setminus \SH$ and 
\[
\mathcal F_2:=\left((\ZZ_S\setminus \QQ_S)^{d-1}\times \mathcal F'\right)\times \left\{\vy=((y^{(p)}_1, \ldots, y^{(p)}_d))_{p\in S} \in \QQ_S^d : y^{(p)}_1\neq 0 \right\}
\]
is a measure-theoretic fundamental domain of $\SH_{\ZZ_S}$ in $\UL_d(\QQ_S)$, i.e., $${\sg_1 \mathcal F_2 \cap \sg_2 \mathcal F_2=\emptyset} \text{ if } \sg_1\neq \sg_2 \in \SH_{\ZZ_S}\;\text{and}$$
\[
\mu^{(d)}_S\left(\UL_d(\QQ_S)- \bigsqcup_{\sg \in \SH_{\ZZ_S}} \sg\mathcal F_2\right)=0.
\]

Hence
\[\begin{split}
d(\lambda_n)^{-d}\xi_1
&=\int_{\mathcal F_2} f(\lambda_n \vy) d\vv' d\mu_S^{(d-1)}(\sg') d\vy \\
&=\int_{\left(\ZZ_S\setminus \QQ_S\right)^{d-1}\times\mathcal F'}\left(\int_{\QQ_S^d}
f(\lambda_n \vy) d\vy\right) d\vv' d\mu_S^{(d-1)}(\sg')\\
&=\int_{\left(\ZZ_S\setminus \QQ_S\right)^{d-1}\times\mathcal F'}
d(\lambda_n)^{-d} I\: d\vv' d\mu_S^{(d-1)}(\sg')\\
&=d(\lambda_n)^{-d} I\:\mu_S^{(d-1)}\left(\UL_{d-1}(\ZZ_S)\setminus \UL_{d-1}(\QQ_S)\right).
\end{split}\]
Therefore $\xi_1=V_{d-1}I$ and similarly, one can show that $\xi_t=(1/t^d)V_{d-1}I$ for any $t\in \NN_S$. Since $\xi=\sum_{t\in \NN_S} \xi_t$, we have
\[V_d=\sum_{t\in \NN_S} \frac 1 {t^d} V_{d-1}=\zeta_S(d) V_{d-1}
\]
and by indcution hypothesis,
\[\mu_S^{(d)}\left(\UL_d(\ZZ_S)\setminus \UL_d(\QQ_S)\right)
=V_d=\prod_{p\in S_f}\left(1-\frac 1 p\right)\cdot\zeta_S(d)\zeta_S(d-1)\cdots \zeta_S(2).
\]

\end{proof}

\section{The First and Second Moment Formulas}
For a bounded compactly supported function $f:\QQ_S^d \rightarrow \RR$, 
define the \emph{Siegel transform} $\widehat f : \AUL_d(\ZZ_S)\setminus\AUL_d(\QQ_S) \rightarrow \RR$ by
\begin{equation}\label{def: Siegel transform}
\widehat f (\Lambda)=\sum_{\vk\in \Lambda} f(\vk).
\end{equation}

In this section, we establish the first and second moment fomulae for $S$-adic Siegel transforms on congruence quotients. 
For $\q{}{}=\q{\sg}{0}$, since
$$\q{}{}\big(q\ZZ_S^d+\vp\big)=\q{}{0}\big((q\ZZ_S^d+\vp)\sg\big)=q^2\q{}{0}\big((\ZZ_S^d+\vp/q)\sg\big),$$
by the duality principle, the space of our interests is 
$$Y_{\vp/q}:=\left\{\left(\ZZ_S^d+\frac {\vp} q \right)\sg : \sg \in \UL_d(\QQ_S)\right\}.
$$

For $q\in \NN_S$, let
\[
\Gamma_1(q)=\left\{\gamma\in \UL_d(\ZZ_S) : \ve_d\gamma \equiv \ve_d\mod q\right\}.
\]

One can show that $Y_{\vp/q}$ can be identified with the quotient space of $\UL_d(\QQ_S)$.

\begin{lemma}\label{GKY Lemma 3.1} 
Let $q\in \NN_S$ and $\vp\in \ZZ_S^d$ be such that $\gcd(q,\vp)=1$. Fix any $\gamma_\vp\in \SL_d(\ZZ_S)$ such that $\vp.\gamma_\vp^{-1}\in \ZZ_S.\ve_d$.
Then $Y_{\vp/q}$ is identified with $\gamma_{\vp}^{-1}\Gamma_1(q)\gamma_{\vp}\setminus \UL_d(\QQ_S)$ via the correspondence $$\left(\ZZ_S^d+\vp/q\right)\sg \leftrightarrow (\gamma_{\vp}^{-1}\Gamma_1(q)\gamma_{\vp})\sg.$$
\end{lemma}
\begin{proof}The proof is identical with that of Lemma 3.1 in \cite{GKY2020}.
\end{proof}

Here, the notion $\gcd(q,\vk)$ can be defined as follow.
\begin{definition}
For $q\in \NN_S$ and $\vk\in \ZZ_S^d$, define
\[
\gcd(q,\vk)=\gcd(q, \vk'),
\]
where $\vk'$ is any integral vector in $(\PP_S\cdot\vk)\cap \ZZ$.
\end{definition}

We now show the first and the second moment formulas for the Siegel transform on $Y_{\vp/q}$. 
Let $\widetilde{\mu}_S$ and $\widetilde{\nu}_S$ be the normalized Haar measures of $\UL_d(\QQ_S)$ and $\SL_d(\QQ_S)$, respectively, such that $$\widetilde{\mu}_S\left(\UL_d(\ZZ_S)\setminus\UL_d(\QQ_S)\right)=1=\widetilde{\nu}_S\left(\SL_d(\ZZ_S)\setminus\SL_d(\QQ_S)\right).$$

\begin{theorem}\label{moment theorems}
Consider $1\neq q\in \NN_S$ and $\vp\in \ZZ_S^d$ such that $\gcd(q, \vp)=1$. Let $f: \QQ_S^d \rightarrow \RR$ be a bounded and compactly supported function.
\begin{enumerate}[(a)]
\item For $d\ge 2$,
\[
\frac 1 {J_q} \int_{Y_{\vp/q}} \widehat f\left(\left(\ZZ_S^d+\frac \vp q\right)\sg\right) d\widetilde{\mu}_S(\sg)
=\int_{\QQ_S^d} f(\vv) d\vv,
\]
where $J_q=\widetilde{\mu}_S(Y_{\vp/q})$.
\item For $d\ge 3$,
\[\begin{split}
&\frac 1 {J_q} \int_{Y_{\vp/q}} \widehat f\left(\left(\ZZ_S^d+\frac \vp q\right)\sg\right)^2 d\widetilde{\mu}_S(\sg)\\
&\hspace{0.4in}=\left(\int_{\QQ_S^d} f d\vv\right)^2+
\sum_{\scriptsize \begin{array}{c}
t\in \NN_S\\
\gcd(t,q)=1\end{array}}\sum_{\scriptsize\begin{array}{c}
a\in q\ZZ_S+t\\
\gcd(a,t)=1\end{array}}
\int_{\QQ_S^d}f(t\vv)f(a\vv)d\vv.
\end{split}\]
\end{enumerate}
\end{theorem}
\begin{remark} One can also consider $Z_{\vp/q}:=\left\{(\ZZ_S^d+\vp/q)\sg : \sg\in \SL_d(\QQ_S)\right\}$ and moment formulas on $Z_{\vp/q}$. By replacing $Y_{\vp/q}$ and $\UL_d(\QQ_S)$ by $Z_{\vp/q}$ and $\SL_d(\QQ_S)$ in the arguments throughout this section, one may obtain the following results: under the same assumptions as in Theorem~\ref{moment theorems}, it follows that
\begin{enumerate}[(a')]
\item For $d\ge 2$,
\[
\frac 1 {J_q} \int_{Z_{\vp/q}} \widehat f\left(\left(\ZZ_S^d+\frac \vp q\right)\sg\right) d\widetilde{\nu}_S(\sg)
=\int_{\QQ_S^d} f(\vv) d\vv,
\]
where $J_q=\widetilde{\nu}_S(Z_{\vp/q})\left(=\widetilde{\mu}_S(Y_{\vp/q})\right)$.
\item For $d\ge 3$,
\[\begin{split}
&\frac 1 {J_q} \int_{Z_{\vp/q}} \widehat f\left(\left(\ZZ_S^d+\frac \vp q\right)\sg\right)^2 d\widetilde{\nu}_S(\sg)\\
&\hspace{0.4in}=\left(\int_{\QQ_S^d} f d\vv\right)^2+
\sum_{\scriptsize \begin{array}{c}
t\in \NN_S\\
\gcd(t,q)=1\end{array}}\sum_{\scriptsize\begin{array}{c}
a\in q\ZZ_S+t\\
\gcd(a,t)=1\end{array}}
\int_{\QQ_S^d}f(t\vv)f(a\vv)d\vv.
\end{split}\]
\end{enumerate} 
\end{remark}

\subsection{Proof of Theorem~\ref{moment theorems}}
Let $X_q=\Gamma(q)\setminus \UL_d(\QQ_S)$, where
\[
\Gamma(q):=\left\{\gamma\in \UL_d(\ZZ_S) : \gamma\equiv \Id\mod q\right\}.
\]
Note that $\Gamma(q)$ is a finite-index subgroup of $\UL_d(\ZZ_S)=\SL_d(\ZZ_S)$ since it is the kernel of the projection
$$\pi_q:\UL_d(\ZZ_S)\rightarrow \UL_d(\ZZ_S/q\ZZ_S)\simeq \SL_d(\ZZ/q\ZZ).$$
Since $\Gamma(q)$ is a normal subgroup of $\UL_d(\ZZ_S)$ and $\Gamma(q)<\Gamma_1(q)$, 
$X_q$ is a finite covering of $Y_q$. 
Hence any function $\phi$ on $Y_q$  extends to a $\gamma_{\vp}^{-1}\Gamma_1(q)\gamma_{\vp}$-invariant function on $X_q$, where $\gamma_{\vp}$ is as in Lemma~\ref{GKY Lemma 3.1} and it immediately follows that 
$$\frac 1 {J_q}\int_{Y_{\vp/q}}\phi\: d\widetilde{\mu}_S=\frac 1 {I_q}\int_{X_q}\phi \:d\widetilde{\mu}_S.$$

\begin{proof}[Proof of Theorem~\ref{moment theorems} (a)]
Note that we assume that $\vp/q\notin \ZZ_S^d$.
 
For a bounded function $f:\QQ_S^d\rightarrow \RR_{\ge 0}$ of compact support and $q\in \NN_S$, define $f_q(\vv):=f(\frac 1 q \vv)$. Since
\[\begin{split}
\widehat{f}\left(\ZZ_S^d+\frac \vp q\right)
&=\sum_{\vv\in \ZZ_S^d+\frac \vp q} f(\vv \sg)
=\sum_{\vv \in \frac 1 q (q\ZZ_S^d+\vp)} f(\vv \sg)
=\sum_{\vk\in q\ZZ_S^d+\vp} f\left(\frac 1 q \vk \sg\right)\\
&\le\sum_{\vk\in \ZZ_S^d-\{\origin\}} f_q(\vk \sg)
=\widehat{f_q}(\ZZ_S^d\sg),
\end{split}\]
by \cite[Lemma 3.8]{HLM2017} and \cite[Lemma 3.10]{HLM2017}, the map 
\[
f\mapsto \frac 1 {I_q}\int_{X_q} \widehat{f}\left((\ZZ_S^d+\vp/q)\sg\right)d\widetilde{\mu}_S(\sg)\] 
is a $\UL_d(\QQ_S)$-invariant positive linear functional on $C_c\big(\prod_{p\in S}(\QQ_p^d-\{\origin\})\big)$,
where $I_q=\widetilde{\mu}_S(X_q)$. 
Since $\UL_d(\QQ_S)$ acts on $\prod_{p\in S}(\QQ_p^d-\{\origin\})$ transitively, by the Riesz-Markov-Kakutani representation theorem, it follows that
\begin{equation}\label{(a) eq:1}
\frac 1 {I_q}\int_{X_q} \widehat{f}\left((\ZZ_S^d+\vp/q)\sg\right)d\widetilde{\mu}_S(\sg)= c\int_{\QQ_S^d} f(\vv) d\vv
\end{equation}
for some $c>0$. Here we use the fact that we can identify a $\UL_d$-invariant measure on $\prod_{p\in S}(\QQ_p^d-\{\origin\})$ with the volume measure of $\QQ_S^d$ up to scaling.

Now, for each $t\in \NN$, let
\[V_t=\QQ_S^d-\left(
\left(B^{(\infty)}_{p^{-t}}(\origin)\times \prod_{p\in S_f} \QQ_p^d\right) 
\cup \bigcup_{i=1}^s 
\left(\RR^d\times p^t \ZZ_{p_i}^d\times \prod_{\scriptsize \begin{array}{c}
p\in S_f\\
p\neq p_i\end{array}}\QQ_p^d\right)\right),
\]
where $B^{(\infty)}_{p^{-t}}(\origin)$ is the ball of radius $p^{-t}$ at the origin in $\RR^d$. For each $k\in \NN$, denote $B_k=B^{(\infty)}_{p^k}(\origin)\times \prod_{p\in S_f} p^{-k}\ZZ_p^d$.
Take a sequence $\{f_k\}$ in $C_c\big(\prod_{p\in S}(\QQ_p^d-\{\origin\})\big)$ such that $f_k\le \ind_{B_k}|_{V_1}$ and as $k\rightarrow \infty$, $f_k$ asymptotically converges to $\ind_{B_k}|_{V_1}$.
By putting $\{\frac 1 {\vol(f_k)} f_k\}$ on both sides in \eqref{(a) eq:1} and letting $k\rightarrow \infty$, we obtain that $c=1$.

For an arbitrary bounded and compactly supported function $f$ on $\QQ_S^d$, since any affine lattice does not intersect with the complement of $\prod_{p\in S}(\QQ_p^d-\{\origin\})$, if we take $\{f'_t\} \subseteq C_c\big(\prod_{p\in S}(\QQ_p^d-\{\origin\}\big)$ such that $f|_{V_{t-1}}\le f'_t\le f|_{V_t}$, then as $t\rightarrow \infty$, $f'_t$ converges to $f$ and we have
\[
\lim_{t\rightarrow\infty}\int_{X_q} \widehat{f'_t}\left(\left(\ZZ_S^d+\frac \vp q\right)\sg\right)d\widetilde{\mu}_S(\sg)= \int_{X_q} \widehat{f}\left(\left(\ZZ_S^d+\frac \vp q\right)\sg\right)d\widetilde{\mu}_S(\sg)
\]
which proves (a) (see \cite[Section 3]{Han2021} for more details).
\end{proof}

Following \cite[Section 3.2]{GKY2020}, we will deduce Theorem~\ref{moment theorems} (b) from Proposition~\ref{MS (7.25) GKY (3.6)} and Proposition~\ref{MS Proposition 7.6}, whose proofs will be provided in Subsection~\ref{MS Section 7-2}.

For each $\vy \in \prod_{p\in S} (\QQ_S^d-\{\origin\})$, define
\[
X_q(\vy)=\left\{\Gamma(q)\sg \in X(q) : \vy \in \left(\ZZ_S^d+\frac \vp q\right)\sg \right\}.
\]

One can assign the probability measure $\nu_\vy$ on $X_q(\vy)$, for each $\vy$, which commutes with right multiplication by an element of $\UL_d(\QQ_S)$ (see Definition~\ref{definition of nu} and Remark~\ref{remark: prob. msr}).

\begin{proposition}\label{MS (7.25) GKY (3.6)} 
Let $d\ge 2$. Let $1\neq q\in \NN_S$ and $\vp \in \ZZ_S^d$ for which $\gcd(q,\vp)=1$. For any Borel measurable function $F:X_q\times \QQ_S^d \rightarrow \RR_{\ge 0}$, it follows that
\[
\frac 1 {I_q}\int_{X_q} \sum_{\vk\in \ZZ_S^d+\vp/q} F(\Gamma(q)\sg, \vk\sg)d\widetilde{\mu}_S(\sg)
=\int_{\QQ_S^d} \int_{X_q(\vy)} F(\Gamma(q)\sg, \vy) d\nu_{\vy}(\sg)d\vy,
\]
where $I_q=\widetilde{\mu}_S(X_q)$.
\end{proposition}

\begin{proposition}\label{MS Proposition 7.6}
Let $f:\QQ_S^d\rightarrow \RR_{\ge 0}$ be a bounded compactly supported function and let $\vy \in \prod_{p\in S} (\QQ_S^d-\{\origin\})$. Then
\[
\int_{X_q(\vy)} \widehat{f}\left(\left(\ZZ_S^d+\frac \vp q \right)\sg\right)\nu_{\vy}(\sg)
=\int_{\QQ_S^d} f d\vv + \hspace{-0.1in} \sum_{\scriptsize \begin{array}{c}
t\in \NN_S\\
\gcd(t,q)=1\end{array}}\hspace{-0.1in}\frac 1 {t^d}
\sum_{\scriptsize \begin{array}{c}
a\in q\ZZ_S+t\\
\gcd(a,t)=1\end{array}}
f\left(\frac a t \vy\right).
\]
\end{proposition}

\begin{proof}[Proof of Theorem~\ref{moment theorems} (b)]
For a given bounded compactly supported function $f:\QQ_S^d\rightarrow \RR_{\ge 0}$, define $$F\left((\ZZ_S^d+\vp/q)\sg, \vy\right):=f(\vy)\widehat{f}((\ZZ_S^d+\vp/q)\sg).$$ Note that the function $F$ can be considered as a function on $Y_{\vp/q}\times \QQ_S^d$ as well as on $X_q\times \QQ_S^d$.

Applying Proposition~\ref{MS (7.25) GKY (3.6)} and Proposition~\ref{MS Proposition 7.6},
\[\begin{split}
&\frac 1 {J_q} \int_{Y_{\vp/q}} \widehat{f}\left(\left(\ZZ_S^d+\frac \vp q\right)\sg\right)^2 d\widetilde{\mu}_S(\sg)\\
&\hspace{0.2in}=\frac 1 {I_q} \int_{X_q} \sum_{\vv \in \ZZ_S^d}f\left(\left(\vv+\frac \vp q\right) \sg\right)\widehat{f}\left(\left(\ZZ_S^d+\frac \vp q\right)\sg\right)d\widetilde{\mu}(\sg)
\\
&\hspace{0.2in}=\int_{\QQ_S^d} f(\vy) \int_{X_q(\vy)} \widehat f\left(\left(\ZZ_S^d+\frac \vp q\right)\sg\right)d\nu_{\vy}(\sg)d\vy\\
&\hspace{0.2in}=\int_{\QQ_S^d} f(\vy) \left(\int_{\QQ_S^d} f(\vx) d\vx 
+ \sum_{\scriptsize \begin{array}{c}
t\in \NN_S \\
\gcd(t,q)=1\end{array}} \frac 1 {t^d} \sum_{\scriptsize \begin{array}{c}
a\in q\ZZ_S+t\\
\gcd(a,t)=1\end{array}} f\left(\frac a t \vy \right)\right) d\vy\\
&\hspace{0.2in}=\left(\int_{\QQ_S^d} f d\vv \right)^2
+\sum_{\scriptsize \begin{array}{c}
t\in \NN_S\\
\gcd(t,q)=1\end{array}}\sum_{\scriptsize\begin{array}{c}
a\in q\ZZ_S+t\\
\gcd(a,t)=1\end{array}} \frac 1 {t^d} \int_{\QQ_S^d} f(\vy) f\left(\frac a t \vy\right) d\vy \\
&\hspace{0.2in}=\left(\int_{\QQ_S^d} f d\vv \right)^2
+\sum_{\scriptsize \begin{array}{c}
t\in \NN_S\\
\gcd(t,q)=1\end{array}}\sum_{\scriptsize\begin{array}{c}
a\in q\ZZ_S+t\\
\gcd(a,t)=1\end{array}}\int_{\QQ_S^d} f(t\vy) f\left(a \vy\right) d\vy.
\end{split}\]
\end{proof}

\subsection{The Spaces $X_q$ and $X_q(\vy)$}\label{MS Section 7-2}
This subsection is a generalization of the part of Section 7 in \cite{MS2010} to $S$-arithmetic spaces.

Note that by definition, $\UL_d(\ZZ_S)=\SL_d(\ZZ_S)$. 
For $q\in \NN_S$, let us define
$$\pi_q:\UL_d(\ZZ_S)\rightarrow \UL_d(\ZZ_S/q\ZZ_S)\simeq \SL_d(\ZZ/q\ZZ)$$
by $\pi_q\left(g=(g_{ij})\right)=([g_{ij}])$, where $[g_{ij}]=g_{ij}+q\ZZ_S\in \ZZ_S/q\ZZ_S$.

Let $\SH$ be the subgroup of $\UL_d(\QQ_S)$ given as in \eqref{dual AUL}.

\begin{proposition}\label{S-arithmetic part} 
Let $1\neq q\in \NN_S$ and $\vp\in \ZZ_S^d$ with $\gcd(q,\vp)=1$. 
There is a subset $\left\{\vk_t \in \ZZ_S^d+\vp/q: t\in \NN_S,\; \gcd(t,q)=1\right\}$ satisfying the following:
\begin{enumerate}[(a)]
\item For any $\vy\in \prod_{p\in S} (\QQ_p^d-\{\origin\})$, 
\[
X_q(\vy)=\bigsqcup_{\scriptsize \begin{array}{c}
t\in \NN_S\\
\gcd(t,q)=1\end{array}} \left\{\Gamma(q)\sg \in \Gamma(q)\setminus\UL_d(\QQ_S): \vk_t\:\sg=\vy \right\}.
\]
\item For $\vk,\; \vy\in \QQ_S^d$, define the set
\[X_q(\vk,\vy):=\left\{\Gamma(q)\sg\in \Gamma(q)\setminus\UL_d(\QQ_S) : \vk\:\sg=\vy \right\}.\] 
Fix any $\sg_\vk$ and $\sg_\vy$ in $\SL_d(\QQ_S)$ for which $\vk=\ve_1\:\sg_\vk$ and $\vy=\ve_1\:\sg_\vy$. Then
\[\begin{split}
X_q(\vk, \vy)
&=\Gamma(q) \sg_\vk^{-1}\left((\sg_\vk \Gamma(q) \sg_\vk^{-1}\cap \SH)\setminus \SH \right)\sg_\vy\\
&\simeq (\sg_\vk \Gamma(q) \sg_\vk^{-1}\cap \SH)\setminus \SH.
\end{split}\]
\end{enumerate}
\end{proposition}
\begin{proof}
(a) We first fix an element $\vy\in \prod_{p\in S} \left(\QQ_p^d-\{\origin\}\right)$. It is easily seen that
\[
\left\{\begin{array}{cl}
X_q(\vk_1, \vy)=X_q(\vk_2, \vy), &\text{if }\vk_1\Gamma(q)=\vk_2\Gamma(q);\\
X_q(\vk_1, \vy)\cap X_q(\vk_2, \vy)=\emptyset, &\text{otherwise.}
\end{array}\right.
\]
Hence
\[
X_q(\vy)=\bigsqcup_{\scriptsize \begin{array}{c}
\vk\Gamma(q)\\
\vk\in \ZZ_S^d+\vp/q\end{array}} X_q(\vk, \vy).
\]

For $\vk \in \ZZ_S^d +\vp/q$, consider $t_\vk:=\gcd(q\vk) \in \NN_S$, i.e., $t_\vk$ is defined to be the unique integer in $\NN_S$ for which $q\vk \in t_\vk\Prim{\ZZ_S^d}$. Since $\gcd(q\vk, q)=\gcd(\vp,q)=1$, it follows that $\gcd(t_\vk,q)=1$. It is obvious that 
\[t_{\vk_1}=t_{\vk_2} \;\text{if}\; \vk_1\Gamma(q)=\vk_2\Gamma(q)
\] 
(in fact, if $\vk_1\UL_d(\ZZ_S)=\vk_2\UL_d(\ZZ_S)$). 

Conversely, we claim that 
\begin{enumerate}[(i)]
\item for each $t\in \NN_S$ such that $\gcd(t,q)=1$, there is $\vk\in \ZZ_S^d+\vp/q$ with $\gcd(q\vk)=t$ and 
\item if $t_{\vk_1}=t_{\vk_2}$, then $\vk_1\Gamma(q)=\vk_2\Gamma(q)$.
\end{enumerate}

\vspace{0.1in}
\noindent (i) For such $t\in \NN_S$, since $\gcd(t,q)=1$, there is $t^*\in \ZZ$ such that $tt^*\equiv 1\mod q$. Pick $\pp\in \PP_S$ such that $\pp\vp\in \ZZ^d$ and $\pp\vp$ is a primitive vector in $\ZZ^d$.
Since $\gcd(q, t^*\pp\vp)=1$, by Dirichlet's theorem on arithmetic progressions, there exists a primitive vector $\vm$ in $t^*\pp\vp+q\ZZ^d$, hence if we let $\vk:=\frac t \pp \vm$, then $\vk \in \ZZ_S^d+\vp/q$ and $\gcd(q\vk)=t$.

\noindent (ii) Let $t=t_{\vk_1}=t_{\vk_2}$. Choose any $\gamma_i\in \SL_d(\ZZ_S)$, $i=1,2$, for which $q\vk_i=t\ve_i\:\gamma_i$. Then since $\vk_1\equiv \vk_2 \mod q$, we see that $\ve_1\:\gamma_1\gamma_2^{-1}\equiv \ve_1 \mod q$ so that $\gamma_1\gamma_2^{-1}$ is of the form
\[
\gamma_1\gamma_2^{-1}=\left(\begin{array}{cc}
x_1 &\vx' \\
\tp\vv & A'
\end{array}\right),
\]
where $x_1\in 1+q\ZZ_S$ and $\vx'\in q\ZZ_S^{d-1}$.
Since $\det \sg'\in 1+q\ZZ_S$, there is $\sg'_1\in \mathfrak M_{d-1}(\{0,1,\ldots, q-1\})$ such that $\sg'_1+\mathfrak M_{d-1}(q\ZZ_S)=\sg'+\mathfrak M_{d-1}(q\ZZ_S)$ and $\det \sg'_1\in (1+q\ZZ_S) \cap \ZZ=1+q\ZZ$.
Hence by \cite[p. 21]{Sim1971}, there is $\sg'_2\in \SL_{d-1}(\ZZ)$ for which $\sg'_2\equiv \sg'_1\mod q$, so that $$\sg' +\mathfrak M_{d-1}(q\ZZ_S)=\sg'_1 +\mathfrak M_{d-1}(q\ZZ_S)= \sg'_2 +\mathfrak M_{d-1}(q\ZZ_S).$$
Since $\SL_d(\ZZ_S/q\ZZ_S)\simeq \SL_d(\ZZ_S)/\Gamma(q)$ and 
$$\gamma_1\gamma_2^{-1}=\left(\begin{array}{cc}
1 & 0 \\
\tp \vv' & \sg'_2 \end{array}\right) \mod q,$$
it follows that there is $\gamma_0\in \Gamma(q)$ such that $\gamma_1\gamma_2^{-1}\gamma_0=\left(\begin{array}{cc}
1 & 0 \\
\tp \vv' & \sg'_2 \end{array}\right)$ or $\gamma_1\gamma'_0\gamma_2^{-1}=\left(\begin{array}{cc}
1 & 0 \\
\tp \vv' & \sg'_2 \end{array}\right)$ for some $\gamma'_0\in \Gamma(q)$ since $\Gamma(q)$ is normal in $\UL_d(\ZZ_S)$.
Hence $\ve_1\gamma_1\gamma_0'=\ve_1\gamma_2,$ so that $\vk_1\Gamma(q)=\vk_2\Gamma(q)$.
Therefore, one can choose $\vk_t\in \ZZ_S^d+\vp/q$, for each $t\in \NN_S$ with $\gcd(t,q)=1$, so that 
$$X_q(\vy)=\bigsqcup_{t\in \NN_S, \gcd(t,q)=1} X_q(\vk_t, \vy).$$

For arbitrary $\vy'\in {\prod_{p\in S}(\QQ_p^d-\left\{\origin\right\})}$, the property (a) also holds for $\{\vk_t\}$ since $X_q(\vk_t, \vy')=X_q(\vk_t,\vy)\sg_{\vy'}$, where $\sg_{\vy'}\in \UL_d(\QQ_S)$ for some $\vy'=\vy\:\sg_{\vy'}$.

\vspace{0.1in}
\noindent (b) This follows in a straightforward fashion from the definitions of $\sg_{\vk}$, $\sg_{\vy}$ and $X_q(\vk,\vy)$.
\end{proof}

\begin{definition}\label{definition of nu}
Let us assign the measure $\nu_{\vy}$ on $X_q(\vy)$ by assigning on each $X_q(\vk_t, \vy)$, the pull-back measure of $\frac 1 {I_q \zeta_S(d)} \widetilde{\mu}_\SH$ on $(\sg_t \Gamma(q) \sg_t^{-1} \cap \SH)\setminus \SH$, where $\sg_t=\sg_{\vk_t}$.
Here, $\widetilde{\mu}_\SH$ is the product measure of the volume measure on $\QQ_S^{d-1}$ and $\widetilde{\mu}_S^{(d-1)}$ on $\UL_{d-1}(\QQ_S)$ such that $\widetilde{\mu}_S^{(d-1)}\left(\SL_{d-1}(\ZZ_S)\setminus\SL_{d-1}(\QQ_S)\right)=1$. 
\end{definition}
If we denote by $R_\sg:X_q \rightarrow X_q$ the map given by right multiplication by $\sg$, it holds that $(R_\sg)_*\nu_\vy=\nu_{\vy\sg}$ for any $\sg \in \UL_d(\QQ_S)$ and $\vy\in \prod_{p\in S}(\QQ_p^d-\{\origin\})$.

\begin{lemma}\label{MS Proposition 7.7}
Let $d\ge 3$. Consider a bounded measurable function $f:\QQ_S^d \rightarrow \RR$ of compact support. For $\eta=(\eta_1, \eta')$ with $\eta_1\in \QQ_S$ and $\eta'\in \ZZ_S^{d-1}$, we have
\[\begin{split}
&\int_{(\Gamma(q)\cap \SH)\setminus \SH}
\widehat{f}\left(\left(\ZZ_S^d + \eta\right)\sh\right)d\widetilde{\mu}_\SH(\sh)\\
&\hspace{1in}= q^{d-1} I_q^{(d-1)}\left(\sum_{\ell\in \ZZ_S} f\left((\ell+\eta_1)\ve_1\right)+ \int_{\QQ_S^d} f d\vv \right),
\end{split}\]
where $I_q^{(d-1)}=\#\UL_{d-1}(\ZZ_S/q\ZZ_S)$.
\end{lemma}
\begin{proof}
Following the notation in Section \ref{Subsect: volume}, let us denote by $\sh=(\vv', \sg')$ an element of $\SH$, where $\vv'\in \QQ_S^{d-1}$ and $\sg'\in \UL_{d-1}(\QQ_S)$. 
Recall that ${(\Gamma(q)\cap \SH)\setminus \SH}\simeq (q\ZZ_S^{d-1}\setminus \QQ_S^{d-1})\times (\Gamma^{(d-1)}(q)\setminus \UL_{d-1}(\QQ_S))$. Let $\mathcal F'$ be a fundamental domain of $\Gamma^{(d-1)}(q)\setminus \UL_{d-1}(\QQ_S)$.

Define $f_1(x_1, \vx')=\sum_{\ell\in \ZZ_S} f(x_1+\ell, \vx')$. Since $\eta'\in \ZZ_S^{d-1}$,
\begin{equation}\begin{split}\label{eq 1:MS Proposition 7.7}
&\int_{(\Gamma(q)\cap \SH)\setminus \SH}
\widehat{f}\left(\left(\ZZ_S^d + \eta\right)\sh\right)d\widetilde{\mu}_\SH(\sh)\\
&=\int_{\mathcal F'}\int_{q\ZZ_S^{d-1}\setminus \QQ_S^{d-1}}
f_1(\eta_1\ve_1)\:d\widetilde{\mu}^{(d-1)}_S(\sg')d\vv'\\
&\hspace{0.1in}+\int_{\mathcal F'}\int_{q\ZZ_S^{d-1}\setminus \QQ_S^{d-1}}\sum_{\vm'\in \ZZ_S^{d-1}-\{\origin\}}
f_1(\eta_1+{\vm'}\:\tp\vv', \vm'\sg')\:d\widetilde{\mu}^{(d-1)}_S(\sg')d\vv'\\
&=q^{d-1}I_q^{(d-1)}f_1(\eta_1\ve_1)\\
&\hspace{0.1in}+\int_{\mathcal F'}\sum_{\vm'\in \ZZ_S^{d-1}-\{\origin\}}
\int_{q\ZZ_S^{d-1}\setminus \QQ_S^{d-1}} f_1(\eta_1+\vm'\:\tp \vv', \vm'\sg')\: d\vv' d\widetilde{\mu}_S^{(d-1)}(\sg').
\end{split}\end{equation}
Note that the integral above can be changed to $I_q^{(d-1)}$ times of the integral over $\UL_{d-1}(\ZZ_S)\setminus \UL_{d-1}(\QQ_S)$, since the function inside is invariant under the action of $\UL_{d-1}(\ZZ_S)$. 
Consider the function on $\QQ_S^{d-1}$ defined by
\[\begin{split}
\vx'\in \prod_{p\in S_f}(\QQ_p^{d-1}-\{\origin\}) &\mapsto 
\int_{q\ZZ_S^{d-1}\setminus \QQ_S^{d-1}} f_1(\eta_1+\vx'\:\tp \vv', \vx')d\vv'\\
&\hspace{0.2in}=q^{d-1}\int_{\QQ_S} f(x_1, \vx')dx_1 
\end{split}\]
and $0$ otherwise. Then the function above is bounded and compactly supported. Applying Theorem \ref{moment theorems} (a) to the last integral in \eqref{eq 1:MS Proposition 7.7}, we obtain the result. 
\end{proof}

\begin{proof}[Proof of Proposition~\ref{MS Proposition 7.6}]
We first observe that for each $\vk_t$ in Proposition~\ref{S-arithmetic part}, one can take $\sg_t=\sg_{\vk_t}$ as $\sa_t\gamma_t$, where $\sa_t=\diag(t/q, q/t, 1, \ldots, 1)$ and $\gamma_t\in \SL_d(\ZZ_S)$. 
Let $\Phi_t : \SH \rightarrow \SH$ be the map given as $\Phi_t(\sh)=\sa_t \sh \sa_t^{-1}$ so that $\sg_\vk \Gamma(q) \sg_\vk^{-1} \cap \SH= \Phi_t(\Gamma(q)\cap \SH)$.
Note that $(\Phi_t)_*\widetilde{\mu}_\SH=(t/q)^d \widetilde{\mu}_\SH$.

By the definition of $\nu_\vy$ and the change of variables by the map $\Phi_t$ for each $t$, we have
\[\begin{split}
&\int_{X_q(\vy)} \widehat{f}\left(\left(\ZZ_S^d+\frac \vp q \right)\sg\right)\nu_\vy(\sg)\\
&=\frac 1 {I_q \zeta_S(d)} \sum_{\scriptsize \begin{array}{c}
t\in \NN_S\\
\gcd(t,q)=1\end{array}}
\int_{(\sg_t \Gamma(q) \sg_t^{-1} \cap \SH)\setminus \SH}
\widehat{f} \left(\left(\ZZ_S^d+\frac \vp q\right)\sg_t^{-1}\sh \sg_\vy\right)d\widetilde{\mu}_\SH(\sh)\\
&=\frac {q^d} {I_q \zeta_S(d)} \sum_{\scriptsize \begin{array}{c}
t\in \NN_S\\
\gcd(t,q)=1\end{array}} \frac 1 {t^d}
\int_{(\Gamma(q) \cap \SH)\setminus \SH}
\widehat{f} \left(\left(\ZZ_S^d+\frac \vp q \gamma_t^{-1}\right)\sh \sa_t^{-1}\sg_\vy\right)d\widetilde{\mu}_\SH(\sh)\\
&=\frac {q^d} {I_q \zeta_S(d)} \sum_{\scriptsize \begin{array}{c}
t\in \NN_S\\
\gcd(t,q)=1\end{array}} \frac 1 {t^d}
\int_{(\Gamma(q) \cap \SH)\setminus \SH}
\widehat{f} \left(\left(\ZZ_S^d+\frac t q \ve_1\right)\sh \sa_t^{-1}\sg_\vy\right)d\widetilde{\mu}_\SH(\sh).
\end{split}\]
Here, the last equality is deduced from the fact that $$(\ZZ_S^d+\vp/q)\gamma_t^{-1}=(\ZZ_S^d+\vk_t)\gamma_t^{-1}=\ZZ_S^d+(t/q)\ve_1$$ by the definition of $\gamma_t\in \SL_d(\ZZ_S)$.

Applying Lemma~\ref{MS Proposition 7.7} with the function $\vv\mapsto f(\vv \sa_t^{-1}\sg_\vy)$ for each $t$, the above integral is
\[\begin{split}
&=\frac {q^{2d-1} I_q^{(d-1)}} {I_q \zeta_S(d)}
\sum_{\scriptsize \begin{array}{c}
t\in \NN_S\\
\gcd(t,q)=1\end{array}} \frac 1 {t^d}
\left(
\sum_{\ell\in \ZZ_S} f\left(\left(\frac {\ell q} t+ 1\right)\vy\right) +
\int_{\QQ_S^d} f d\widetilde{\mu}_S\right)\\
&=\frac {q^{2d-1} I_q^{(d-1)}} {I_q \zeta_S(d)}
\hspace{-0.06in}\left(\hspace{-0.05in}\sum_{\scriptsize \begin{array}{c}
t\in \NN_S\\
\gcd(t,q)=1\end{array}} \hspace{-0.15in}\frac 1 {t^d}\right)\hspace{-0.1in}
\left(
\sum_{\scriptsize \begin{array}{c}
t\in \NN_S\\
\gcd(t,q)=1\end{array}}\hspace{-0.15in}\frac 1 {t^d}\sum_{\scriptsize \begin{array}{c}
a\in q\ZZ_S+t\\
\gcd(a,t)=1\end{array}} \hspace{-0.1in}f\left(\frac {a} t\vy\right) +
\int_{\QQ_S^d} f d\widetilde{\mu}_S\right),
\end{split}\]
where we put $t_1=\gcd(\ell q+t, q)$, $t_2=t/t_1$ and $a=(\ell q+t)/t_1$ in the inner summation and relabel them. 

As in the proof of Proposition 7.6 in \cite{MS2010}, since $I_q=q^{2d-1}I_1^{(d-1)}\sum_{e|q}\mu(e)e^{-d}$ and $\sum_{t\in \NN_S, \gcd(t,q)=1} 1/t^d=\zeta_S(d)\sum_{e|q} \mu(e)e^{-d}$, where $\mu(e)$ is the Mobius function, 
\[
\frac {q^{2d-1} I_q^{(d-1)}} {I_q \zeta_S(d)}
\cdot\sum_{\scriptsize \begin{array}{c}
t\in \NN_S\\
\gcd(t,q)=1\end{array}} \frac 1 {t^d}
=1
\]
whereby we deduce Proposition \ref{MS Proposition 7.6}.
\end{proof}

\begin{remark}\label{remark: prob. msr} Using $(\Phi_t)_*\widetilde{\mu}_\SH=(t/q)^d \widetilde{\mu}_\SH$, with the similar argument as in the proof of Proposition \ref{MS Proposition 7.6}, one can show that $\nu_{\vy}$ is the probability measure on $X_q(\vy)$ (see \cite[Proposition 7.5]{MS2010}).
\end{remark}

\begin{proof}[Proof of Proposition~\ref{MS (7.25) GKY (3.6)}]
As mentioned before, since $\prod_{p\in S}(\QQ_p^d-\{\origin\})$ is a conull set in $\QQ_S^d$, it suffices to show that
\begin{equation*}\label{eq 0: MS (7.25) GKY (3.6)}
\begin{split}
&\frac 1 {I_q}\int_{X_q} \sum_{\vk\in \ZZ_S^d+\vp/q} F(\Gamma(q)\sg, \vk\sg)d\widetilde{\mu}_S(\sg)\\
&\hspace{1.2in}=\int_{\prod_{p\in S}(\QQ_p^d-\{\origin\})} \int_{X_q(\vy)} F(\Gamma(q)\sg, \vy) d\nu_{\vy}(\sg)d\vy.
\end{split}\end{equation*}

We may assume that $F=\ind_\mathcal U\cdot \ind_W$, where $\mathcal U\subseteq \Gamma(q)\setminus \UL_d(\QQ_S)$ and $W \subseteq \QQ_S^d$ are bounded and measurable, and show that
\begin{equation}\label{eq 1:MS (7.25) GKY (3.6)}\begin{split}
&\int_{W\cap \prod_{p\in S}(\QQ_p^d-\{\origin\})} \int_{X_q(\vy)} \ind_{\mathcal U} (\Gamma(q)\sg) d\nu_\vy(\sg)d\vy\\
&\hspace{1in}=\frac 1 {I_q}  \sum_{\vk\in \ZZ_S^d+\vp/q} \int_{X_q}
\ind_\mathcal U(\Gamma(q)\sg) \ind_W(\vk \sg) d\widetilde{\mu}_S(\sg).
\end{split}\end{equation}

Let $\mathcal U_0\subseteq \UL_d(\QQ_S)$ be the preimage of $\mathcal U$ under the projection $$\UL_d(\QQ_S)\rightarrow \Gamma(q)\setminus \UL_d(\QQ_S).$$
Let $\mathcal F$ be a fundamental domain of $\Gamma(q)\setminus \UL_d(\QQ_S)$. 
For each $t\in \NN$ with $\gcd(t,q)=1$, let $\vk_t$, $\sg_t=\sg_{\vk_t}$ be as in Proposition~\ref{S-arithmetic part}. Denote by $R_t$ the set of representatives for $(\Gamma(q)\cap \sg_t^{-1}\SH\sg_t)\setminus \Gamma(q)$.
Then one can check that
\begin{enumerate}[(i)]
\item for each $t\in \NN_S$ with $\gcd(t,q)=1$, $\big(\sg_t (\bigsqcup_{\gamma\in R_t} \gamma \mathcal F) \sg_\vy^{-1}\big)\cap \SH$ is a fundamental domain of $(\sg_t \Gamma(q) \sg_t^{-1} \cap \SH )\setminus \SH$;
\item there is a one-to-one correspondence between $\bigsqcup_{t\in \NN_S, \gcd(t,q)=1} \vk_t\:R_t$ and $\ZZ_S^d+\vp/q$
\end{enumerate}
(see \cite[Proof of Proposition 7.3]{MS2010}). 

By the definition of $\nu_\vy$,
\[\begin{split}
&\int_{X_q(\vy)} \ind_{\mathcal U} (\Gamma(q)\sg) d\nu_\vy(\sg)\\
&=\frac 1 {I_q \zeta_S(d)} \hspace{-0.15in}\sum_{\scriptsize \begin{array}{c}
t\in \NN_S\\
\gcd(t,q)=1\end{array}}\hspace{-0.15in}
\int_{(\sg_t\Gamma(q)\sg_t^{-1}\cap \SH)\setminus \SH}
\ind_{\mathcal U} (\Gamma(q)\sg_t^{-1}\left((\sg_t\Gamma(q)\sg_t^{-1}\cap\SH)\sh\right)\sg_\vy) d\widetilde{\mu}_\SH(\sh).
\end{split}\]

For each $\gamma\in R_t$ and $\sh=\sg_t(\gamma \sg')\sg_\vy^{-1}\in \sg_t(\gamma \mathcal F)\sg_\vy^{-1}\cap \SH$, 
\[
\sg_t^{-1}\sh\sg_\vy=\gamma\sg'\in \mathcal U_0\;
(\Leftrightarrow \gamma\sg'\in \gamma (\mathcal U_0\cap \mathcal F))\;\Leftrightarrow\;
\sh \in \sg_t\gamma(\mathcal U_0\cap \mathcal F) \sg_\vy^{-1}.
\] 
Hence
\begin{equation}\label{eq 2:MS (7.25) GKY (3.6)}\begin{split}
&\int_{X_q(\vy)} \ind_{\mathcal U} (\Gamma(q)\sg) d\nu_\vy(\sg)\\
&\hspace{0.5in}=\frac 1 {I_q \zeta_S(d)}\sum_{\scriptsize \begin{array}{c}
t\in \NN_S\\
\gcd(t,q)=1\end{array}}\sum_{\gamma\in R_t}
\int_{\SH} \ind_{\sg_t\gamma (\mathcal U_0\cap \mathcal F) \sg_{\vy}^{-1}}(\sh)d\widetilde{\mu}_H(\sh)\\
&\hspace{0.5in}=\frac 1 {I_q \zeta_S(d)} \sum_{\vk\in \ZZ_S^d+\vp/q}
\int_{\SH} \ind_{\sg_{\vk} (\mathcal U_0\cap \mathcal F) \sg_{\vy}^{-1}}(\sh)d\widetilde{\mu}_H(\sh),
\end{split}\end{equation}
where we take $\sg_\vk=\sg_t\gamma$ for each $t$ and $\gamma\in R_t$, which are associated with $\vk$ according to (ii).

\vspace{0.15in}
On the other hand, in view of Proposition~\ref{Def Vol} and Theorem~\ref{Com Vol}, we have
\[
\int_{\UL_d(\QQ_S)} f(\sg) d\widetilde{\mu}_S(\sg)
=\frac 1 {\zeta_S(d)} \int_{\QQ_S^d} \int_{\SH} f(\sh \sg_\vy)d\widetilde{\mu}_\SH(\sh)d\vy.
\]

Applying the above equation to each integral in the right hand side of \eqref{eq 1:MS (7.25) GKY (3.6)}, since $\widetilde{\mu}_S$ is $\UL_d(\QQ_S)$-invariant, we have
\begin{equation}\label{eq 3:MS (7.25) GKY (3.6)}\begin{split}
&\int_{X_q} \ind_{\mathcal U}(\Gamma(q)\sg)\ind_W(\vk\sg)d\widetilde{\mu}_S(\sg)\\
&\hspace{0.5in}=\int_{\UL_d(\QQ_S)} \ind_{\mathcal F\cap \mathcal U_0}(\sg_{\vk}^{-1}\sg)\ind_W(\vk\sg_{\vk}^{-1}\sg)d\widetilde{\mu}_S(\sg)\\
&\hspace{0.5in}=\frac 1 {\zeta_S(d)} \int_{\QQ_S^d} \int_{\SH}
\ind_{\mathcal F\cap\mathcal U_0}(\sg_\vk^{-1}\sh\sg_\vy)\ind_W(\vk\:\sg_\vk^{-1}\sh\sg_\vy)d\widetilde{\mu}_\SH(\sh)d\vy.
\end{split}\end{equation}

Since $\vk\sg_\vk^{-1}\sh\sg_\vy=\vy$, the equation \eqref{eq 1:MS (7.25) GKY (3.6)} follows from \eqref{eq 2:MS (7.25) GKY (3.6)} and \eqref{eq 3:MS (7.25) GKY (3.6)}.
\end{proof}

\section{Proof of Theorem~\ref{congruence case}}

From \eqref{relation}, Theorem~\ref{congruence case} is a direct consequence of the theorem below. Recall the notation $\NT(\q{}{\xi}, \mathcal I, \T)$ in Theorem~\ref{inhomogeneous case}.

\begin{theorem}\label{reduced congruence case} 
Under the same assumptions as in Theorem~\ref{congruence case}, there is $\delta_0=\delta_0(d,\kappa)>0$ such that for any $\delta\in(0,\delta_0)$, we have 
\[
\NT(\q{}{\frac {\vp} q}, \mathcal I, \T)
=c_{\q{}{}} \vol(\I_\T)|\T|^{d-2} + o\left(|\T|^{d-2-\kappa-\delta}\right)
\]
for almost every unimodular non-degenerate isotropic quadratic form $\q{}{}$.
Here the implied constant of the error term is uniform on a comapct set of $Y_{\vp/q}$.
\end{theorem}

\begin{proof}[Proof of Theorem~\ref{congruence case}]
By \eqref{relation}, 
\[
\NT(q,\vp; \q{}{}, \mathcal I, \T)=\NT(\q{}{\frac {\vp} q}, \mathcal I', \T'),
\]
where $\T'=\left(T_\infty/q, T_{p_1}, \ldots, T_{p_s}\right)$ and $\mathcal I'=\{\I'_{\T'}\}$ with $\I'_{\T'}=\frac 1 {q^2}\I_\T$.
Note that $|\T'|=\frac 1 q |\T|$ and $\vol(\I'_{\T'})=\frac 1 {q^2} \vol(\I_\T)$. In particular, $\mathcal I'$ satisfies that $\vol(\I'_{\T'})=(c/{q^{2-\kappa}})|\T'|^{-\kappa}$. By Theorem~\ref{reduced congruence case}, 
\[\begin{split}
\NT(q,\vp;\q{}{},\mathcal I,\T)
&=\NT(\q{}{\frac {\vp} q}, \mathcal I', \T')
=c_{\q{}{}}\vol(\I'_{\T'})|\T'|^{d-2}+o\left(|\T'|^{d-2-\kappa-\delta}\right)\\
&=c_{\q{}{}}\frac {1}{q^d}\vol(\I_\T)|\T|^{d-2}+o(|\T|^{d-2-\kappa-\delta}).
\end{split}\]
\end{proof}

\begin{theorem}\label{variance property}
Let $d\ge 3$. Let $A=\prod_{p\in S} A_p$ be the product of bounded Borel sets $A_p$ in $\QQ_p^d$ for each $p\in S$. There is a constant $C_d>0$, depending only on the dimension $d$, such that
\[
\widetilde{\mu}_S\left(\left\{
\Lambda \in Y_{\vp/q} : \left|\#(\Lambda \cap A) - \vol(A)\right|>M
\right\}\right)< J_qC_d\cdot\frac {\vol(A)} {M^2}.
\]
\end{theorem}
\begin{proof}
Let $\ind_A$ be the indicator function of $A\in \QQ_S^d$.
Since
\[\begin{split}
&\{(t,a): t\in \NN_S, \;\gcd(t,q)=1, \; a\in q\ZZ_S+t, \; \gcd(a,t)=1\}\\
&\hspace{1in}\subseteq \{(t,a): t\in \NN_S,\; a\in q\ZZ_S-\{\origin\}, \; \gcd(a,t)=1\},
\end{split}\]
by Theorem \ref{moment theorems} and the proof of Proposition 4.2 (b) in \cite{Han2021}, there is $C_d>0$ for which
\begin{equation}\label{variance property: eq (1)}
\frac 1 {J_q} \int_{Y_{\vp/q}} {\widehat{\ind_A}\:}^2 d\widetilde{\mu}_S
\le \vol(A)^2 + C_d \vol(A).
\end{equation}

The result follows from \eqref{variance property: eq (1)} 
and Chebyshev's inequality with the probabily space $\left(Y_{\vp/q}, \frac 1 {J_q} \widetilde{\mu}_S\right)$. 
\end{proof}

\begin{theorem}\label{volume formula}
Let $d\ge 3$. For a given isotropic quadratic form $\q{}{}=(\q{p}{})_{p\in S}$, let $\sg=(g_p)_{p\in S}\in \GL_d(\QQ_S)$ be such that $(\q{p}{})^{g_p}$ is of the form $2x_1x_d+(\q{p}{})'(x_2, \ldots, x_{d-1})$, where the coefficients of $(\q{p}{})'$ are in $\ZZ_p$ if $p\neq 2$ and in $2\ZZ_p$ if $p=2$.
For each $p\in S_f$, denote by $k^{(p)}_0$, $z^{(p)}$ integers satisfying that 
\[
g_p (\ZZ_p^d - p\ZZ_p^d)+p^{k^{(p)}_0}\ZZ_p^d=g_p (\ZZ_p^d - p\ZZ_p^d) \quad\text{and}\quad
p^{z^{(p)}}\ZZ_p^d \subseteq g_p (\ZZ_p^d - p\ZZ_p^d).
\] 

Let $\I=(I^{(p)})_{p\in S}\subseteq [-N,N]\times \prod_{p\in S_f} p^{b^{(p)}} \ZZ_p^d$ be an $S$-interval such that for each $p\in S_f$, there is $k_p\in \ZZ$ for which
$I^{(p)}+p^{k^{p}_1}\ZZ_p=I^{(p)}.$

Then for $\T$ with $T_\infty> 2N^{1/d}$ and 
\begin{equation}\label{local property}
2t_p\ge \left\{\begin{array}{lc}
\max\{1+k^{(p)}_0+z^{(p)}-b^{(p)}, 1+k^{(p)}_1+2z^{(p)}-2b^{(p)}\}, &
\text{if } p\neq 2;\\ 
\max\{1+k^{(p)}_0+z^{(p)}-b^{(p)}+1, 1+k^{(p)}_1+2z^{(p)}-2b^{(p)}+2\}, & \text{if } p=2,
\end{array}\right. 
\end{equation}
we have that 
\[
\vol\left(\q{-1}{}(\I)\cap B(\origin, \T)\right)
=c_{\q{}{}} \vol(\I) |\T|^{d-2}+o_{\q{}{}}\left(\vol(\I)|\T|^{d-2}\right).
\]
\end{theorem}
\begin{proof} Since the volume is the product measure $\prod_{p\in S}\vol_p$, it suffices to show the formula for each $p\in S$. 
For the real case, see \cite[Theorem 5]{KY2018}. 
For the $p$-adic case, the proof of Proposition 4.2 in \cite{HLM2017} implies the statement for $p\ge 3$. 
In the case when $p=2$, almost the same proof as that of Proposition 4.2 in \cite{HLM2017} is applicable for $\q{p}{}$, which is of the form $2x_1x_d+q'(x_2,\ldots, x_{d-1})$, where the coefficients of $q'$ are in $2\ZZ_p$.
\end{proof}

For a discrete set $\Lambda$ and a finite-volume set $A$ in $\QQ_S^d$, define
\[
D(\Lambda, A)=\left|\#(\Lambda \cap A) - \vol(A) \right|.
\]
One can obtain the following lemma directly.

\begin{lemma}\label{discrepancy}
Let $\Lambda\subseteq \QQ_S^d$ be a discrete set. Let $A_1 \subseteq A \subseteq A_2 \subseteq \QQ_S^d$ be sets with finite volume. Then
\[
D(\Lambda, A)+\vol(A_2-A_1)\le \max\left\{D(\Lambda, A_1), D(\Lambda, A_2)\right\}.
\]
\end{lemma}

\begin{proof}[Proof of Theorem~\ref{reduced congruence case}]
We will follow the strategy of the proof of Theorem 2.10 in \cite{Han2021}, which is based on the Borel-Cantelli lemma. 
We first fix an arbitrary compact set $\mathcal K$ in $\UL_d(\QQ_S)$.
For notational simplicity, let us denote $$\Lambda_0=\ZZ_S^d+\frac {\vp} q.$$

For a quadratic form $\q{}{}$, $\I\subseteq \QQ_S$ and $\T$, define 
\[
A_{\q{}{},\I, \T}=\q{-1}{}(\I)\cap B_{\T},
\]
where $B_{\T}=\{\vv\in \QQ_S^d : \|\vv\|_p<B_{\T}\}$.

Let $(\delta_p)_{p\in S}$ be an $S$-tuple of positive real numbers.
For $\J=(j_\infty, p_1^{j_1}, \ldots, p_s^{j_s})\in \NN\times \prod_{p\in S_f} p^{\NN}$,
define
\[
\mathcal C_{\J}
=\left\{
\sg\in \mathcal K: \begin{array}{c}
\hspace{-0.4in}D\left(\Lambda_0, A_{\q{\sg}{0},\T}\right)>\vol(\I_\T)
\prod_{p\in S} T_p^{d-2-\delta_p}\\[0.05in]
\hspace{1.2in}\text{for some }\T \in [j_\infty, j_\infty+1)\times \J^f\end{array}
\right\}.
\]

The theorem follows if $\widetilde{\mu}_S\left(\bigcap_{\J_0}\bigcup_{\J\succeq \J_0}\mathcal C_\J\right)=0$ and by Borel-Cantelli lemma, it suffices to show that
\begin{equation}\label{main thm: eq (1)}
\sum_{\J \succeq \J_0} \widetilde{\mu} (\mathcal C_\J)<\infty.
\end{equation}
for appropriate $(\delta_p)_{p\in S}$. Let us choose $(\delta_p)_{p\in S}$, $(\alpha_p)_{p\in S}$ and $(\beta_p)_{p\in S}$ for which
\begin{equation}\label{B-C condition}\begin{split}
&\left\{\begin{array}{l}
0<\delta_\infty < \alpha_\infty,\\
\delta_\infty < \beta_\infty +1,\\
-(d-2-\kappa_\infty)-2\delta_\infty+\alpha_\infty\left(\frac 1 2 (d+2)(d-1)\right)+\beta_\infty < -1;
\end{array}
\right.\\[0.1in]
&\left\{\begin{array}{l}
0<\delta_p < \alpha_p,\\
\delta_p < \beta_p,\\
-(d-2-\kappa_p)-2\delta_p+\alpha_p\left(\frac 1 2 (d+2)(d-1)\right)+\beta_p < 0.
\end{array}
\right.
\end{split}\end{equation}
Note that the range of such $\delta_p$ are $\left(0, \frac {d-2-\kappa_p} {(d+2)(d-1)/2-1}\right)$. 

\vspace{0.2in}
\noindent\emph{The first Approximation: the space}\hspace{0.2in}
We first observe that for $\sh$ and $\sg=(g_p)_{p\in S}$ in $\UL_d(\QQ_S)$, $\I\subseteq \QQ_S$ and $\T=(T_p)_{p\in S}$,
\[\begin{split}
&\vv \in \Lambda_0: \;
\q{\sh\sg}{0}:\; \q{\sh\sg}{0}(\vv)\in \I, \;\|\vv\|_p<T_p\;(p\in S)\\
&\Rightarrow \vw\in \Lambda_0\sg:\;
\q{\sh}{0}:\; \q{\sh}{0}(\vw)\in \I, \;\|\vw\|_p<\|g_p\|_{op}T_p\;(p\in S).
\end{split}\]

For each $\J$, let $\varepsilon_1=\varepsilon_1(\J)=j^{-\alpha_\infty}\cdot\prod_{p\in S_f}p^{-\alpha_p j_p}$. 
One can find $C(\mathcal K)>0$ such that for each $\J$, there is a subset $\mathcal Q=\mathcal Q(\mathcal K, \J)$ of $\mathcal K$ for which
\begin{enumerate}[(i)]
\item $\mathcal K \subseteq \bigcup_{\sh\in \mathcal Q} \sh.\mathcal B(\varepsilon_1)$, where
\[
\mathcal B(\varepsilon_1)=
\left\{
g_\infty \in \SL_d(\RR) : \|g_\infty\|_{op}\le 1+\varepsilon_1
\right\}\times \prod_{p\in S_f} \UL_d(\ZZ_p);
\]
\item $\# \mathcal Q(\mathcal K, \J) < C(\mathcal K)\varepsilon^{-\frac 1 2 (d+2)(d-1)}.$
\end{enumerate}

Here, $\frac 1 2 (d+2)(d-1)$ is the codimension of $\SO(d)$ in $\SL_d(\RR)$ since elements in $\SO(d)\times \prod_{p\in S_f}\UL_d(\ZZ_p)$ have unit operator norms.

For $\sg \in \mathcal B(\varepsilon_1)$,
since 
\[A_{\q{\sh}{0}, \I_\T, (T_\infty(1-\varepsilon_1), \J^f)}
\subseteq A_{\q{\sh\sg}{0}, \I_\T, (T_\infty, \J^f)} 
\subseteq A_{\q{\sh}{0}, \I_\T, (T_\infty(1+\varepsilon_1), \J^f)},\]
if we put $\I^u_{\T}:=\I_{(T_\infty(1+\varepsilon)^{-1}, \J^f)}$ and 
$\I^\ell_{\T}:=\I_{(T_\infty(1-\varepsilon)^{-1}, \J^f)}$, where $\mathcal I=\{\I_\T\}$,
it follows that for all sufficiently large $\J$ (depending on the choice of $(\alpha_p)_{p\in S}$ and $(\delta_p)_{p\in S}$), by Lemma~\ref{discrepancy} and Theorem~\ref{volume formula} with the condition in \eqref{B-C condition},
\begin{equation*}
\mathcal C_\J \subseteq C_\J^u \cup C_\J^\ell,
\end{equation*}
where
\[\begin{split}
C_\J^u&=\bigcup_{\sh\in \mathcal Q}
\left\{
\sg\in \mathcal B(\varepsilon_1): \begin{array}{c}
D\left(\Lambda_0\sg, A_{\q{\sh}{0},\I^u_\T,\T}\right)>0.99\vol(\I^u_\T)
\prod_{p\in S} T_p^{d-2-\delta_p}\\[0.05in]
\hspace{0.8in}\text{for some }\T \in [j_\infty, j_\infty+1+\varepsilon_1)\times \J^f\end{array}
\right\},
\\
C_\J^\ell&=\bigcup_{\sh\in \mathcal Q}
\left\{
\sg\in \mathcal B(\varepsilon_1): \begin{array}{c}
D\left(\Lambda_0\sg, A_{\q{\sh}{0},\I^\ell_\T,\T}\right)>0.99\vol(\I^\ell_\T)
\prod_{p\in S} T_p^{d-2-\delta_p}\\[0.05in]
\hspace{0.8in}\text{for some }\T \in [j_\infty-\varepsilon_1, j_\infty+1)\times \J^f\end{array}
\right\}.
\end{split}\]

We remark that since $I^{(\infty)}_{T_\infty}$ is a sequence of decreasing interval and $\kappa_p<2$ for any $p\in S_f$, one can find a uniform $\T_0$, depending on the compact set $\mathcal K$, such that Theorem~\ref{volume formula} holds for $\q{\sh}{}$ for any $\sh \in \mathcal K$ and $\T\succeq \T_0$.

\noindent \emph{The second Approximation: the radius}\hspace{0.2in}
Let $\varepsilon_2=j_\infty^{-\beta_\infty}\prod_{p\in S_f}j_p^{-\beta_p j_p}$.
For each $k=0, 1, \ldots, \lfloor (1+\varepsilon_1)/\varepsilon_2 \rfloor$, and for $j_\infty+\varepsilon_2 k\le T_\infty < j_\infty+\varepsilon_2(k+1)$, we have
\[
A_{\q{\sh}{0}, \I^u_{\J_{k+1}}, \J_k}
\subseteq
A_{\q{\sh}{0}, \I^u_\T, \T}
\subseteq
A_{\q{\sh}{0}, \I^u_{\J_k}, \J_{k+1}},
\]
where $\J_k=(j_\infty+\varepsilon_2 k, \J^f)$.
Again,  by Lemma~\ref{discrepancy} and Theorem~\ref{volume formula} with the condition in \eqref{B-C condition} provided that 
$$(j_\infty+\varepsilon_2 k)^{-\kappa_\infty}-(j_\infty+\varepsilon_2(k+1))^{-\kappa_\infty} \ll_{\kappa_\infty} (j_\infty+\varepsilon_2k)^{-\kappa_\infty-1},$$ 
it follows that for all sufficiently large $\J$,
\[
\mathcal C^u_\J \subseteq
\mathcal C^{uu}_\J \cup \mathcal C^{u\ell}_\J,
\]
where for each $\J_k$,
\[
\I^{uu}_{\J_k}:=\I^u_{(j_\infty+\varepsilon_2(k-1), \J^f)}\quad\text{and}\quad\I^{u\ell}_{\J_k}:=\I^u_{(j_\infty+\varepsilon_2(k+1), \J^f)},\] 
and
\[\begin{split}
\mathcal C^{uu}_\J&=\bigcup_{\sh\in \mathcal Q}\bigcup_{k=1}^{\lfloor \frac{1+\varepsilon_1} {\varepsilon_2}\rfloor+1}
\left\{\sg \in \mathcal B(\varepsilon_1) : \begin{array}{l}
D\left(\Lambda_0\sg, A_{\q{\sh}{0}, \I^{uu}_{\J_k}, \J_k}\right)
\hspace{0.1in}> 0.9 \vol(\I^{uu}_{\J_k}) \\[0.1in]
\hspace{0.4in}\times(j_k+\varepsilon_2 k)^{d-2-\delta_\infty}\prod\limits_{p\in S_f} p^{j_p(d-2-\delta_p)}
\end{array}
\right\},\\
\mathcal C^{u\ell}_\J&=\bigcup_{\sh\in \mathcal Q}\bigcup_{k=0}^{\lfloor \frac{1+\varepsilon_1} {\varepsilon_2}\rfloor}
\left\{\sg \in \mathcal B(\varepsilon_1) : \begin{array}{l}
D\left(\Lambda_0\sg, A_{\q{\sh}{0}, \I^{u\ell}_{\J_k}, \J_k}\right)>0.9 \vol(\I^{u\ell}_{\J_k})\\[0.1in]
\hspace{0.4in} \times (j_k+\varepsilon_2 k)^{d-2-\delta_\infty}\prod\limits_{p\in S_f} p^{j_p(d-2-\delta_p)}
\end{array}
\right\}.
\end{split}\]
\end{proof}

Finally, by Theorem~\ref{variance property} and Theorem~\ref{volume formula},
for each $k$,
\[\begin{split}
&\widetilde{\mu}_S\left(\left\{\sg \in \mathcal B(\varepsilon_1) : \begin{array}{l}
D\left(\Lambda_0\sg, A_{\q{\sh}{0}, \I^{uu}_{\J_k}, \J_k}\right)
\hspace{0.1in}> 0.9 \vol(\I^{uu}_{\J_k}) \\[0.1in]
\hspace{0.4in}\times(j_k+\varepsilon_2 k)^{d-2-\delta_\infty}\prod\limits_{p\in S_f} p^{j_p(d-2-\delta_p)}
\end{array}\right\}\right)\\
&\hspace{2.3in}\ll 
j_\infty^{\kappa_\infty-(d-2)-2\delta_\infty}
\prod_{p\in S_f} p^{j_p(\kappa_p-(d-2)-2\delta_p)}
\end{split}\]
so that
\[\begin{split}
&\widetilde{\mu}_S(\mathcal C^{uu}_\J) \ll
j_\infty^{-(d-2-\kappa_\infty)-2\delta_\infty+\alpha_\infty\left(\frac 1 2 (d+2)(d-1)\right)+\beta_\infty}\\
&\hspace{1.8in}\cdot\prod_{p\in S_f}
p^{j_p(-(d-2-\kappa_p)-2\delta_p+\alpha_p\left(\frac 1 2 (d+2)(d-1)\right)+\beta_p)}.
\end{split}\]
Hence $\sum_{\J\succeq \J_0} \widetilde{\mu}_S(\mathcal C^{uu}_\J)<\infty$ by \eqref{B-C condition} for sufficiently large $\J_0$. Similarly, one can show that the summands of $\widetilde{\mu}_S(\mathcal C^{u\ell})$, $\widetilde{\mu}_S(\mathcal C^{\ell u})$ and $\widetilde{\mu}_S(\mathcal C^{\ell\ell})$ are finite, which shows \eqref{main thm: eq (1)}.

\section{The Space of Inhomogeneous Quadratic Forms}

Using Rogers' higher moment formulas for the space of unimodular affine lattices in $\RR^d$ (\cite[Appendix B]{EMV2015}, \cite[Lemma 4]{A15}), in \cite{GKY2020}, it was noted that the effective Oppenheim conjecture holds for almost all unimodular affine lattices in $\RR^d$. In this section, we generalise this result to the space of unimodular affine $S$-lattices in $\QQ_S^d$ and for this, let us first show Rogers' higher moment formulas for $\AUL_d(\ZZ_S)\setminus \AUL_d(\QQ_S)$.

\begin{theorem}\label{moment formulae: inhomogeneous} 
\begin{enumerate}[(a)] Let $d\ge 2$.
\item  For a bounded compactly supported function $f:\QQ_S^d\rightarrow \RR_{\ge 0}$, 
\[
\int_{\AUL_d(\ZZ_S)\setminus \AUL_d(\QQ_S)} \widehat{f}\: d\widetilde{\mu}_Sd\vv
=\int_{\QQ_S^d} f(\vv) \:d\vv. 
\]
\item For a bounded compactly supported function $F:(\QQ_S^d)^2\rightarrow \RR_{\ge 0}$, 
\[\begin{split}
&\int_{\AUL_d(\ZZ_S)\setminus \AUL_d(\QQ_S)} 
\sum_{\vk_1, \vk_2\in \ZZ_S^d} F\left(\vk_1\sg+\vv, \vk_2\sg+\vv\right) d\widetilde{\mu}_S d\vv \\
&\hspace{1.5in}=\int_{(\QQ_S^d)^2} F(\vv_1,\vv_2) d\vv_1d\vv_2+\int_{\QQ_S^d} F(\vv,\vv) d\vv.
\end{split}\]
In particular, if we let $F(\vx_1,\vx_2)={\ind_A}(\vx_1)\ind_A(\vx_2)$ for a borel set $A\subseteq \QQ_S^d$, it holds that
\[
\int_{\AUL_d(\ZZ_S)\setminus \AUL_d(\QQ_S)}
{\widehat{\ind_A}\:}^2 d\widetilde{\mu}_S d\vv=\vol(A)^2 + \vol(A).
\] 
\end{enumerate}
\end{theorem}
\begin{proof}
Let $\mathcal F$ be a fundamental domain of $\UL_d(\ZZ_S)\setminus \UL_d(\QQ_S)$. Then a fundamental domain of $\AUL_d(\ZZ_S)\setminus \AUL_d(\QQ_S)$ is $\bigcup_{\sg \in \mathcal F} \mathcal T\sg$, where $\mathcal T=[0,1)^d\times \prod_{p\in S_f} \ZZ_p^d$ is a fundamental domain of $\ZZ_S^d\setminus \QQ_S^d$.

\vspace{0.1in}
\noindent (a) By the change of variables $\vw=\vv\sg^{-1}$ and letting $f_\sg(\vv):=f(\vv\sg)$ for each $\sg \in \mathcal F$, we have
\[\begin{split}
&\int_{\AUL_d(\ZZ_S)\setminus \AUL_d(\QQ_S)} \widehat{f}\: d\widetilde{\mu}_Sd\vv
=\int_{\mathcal F}\int_{\mathcal T\sg} 
\sum_{\vk\in \ZZ_S^d} f(\vk\sg+ \vv)\: d\vv d\widetilde{\mu}_S(\sg)\\
&=\int_{\mathcal F}\int_{\mathcal T}
\sum_{\vk\in \ZZ_S^d} f((\vk+\vw)\sg) d\vw d\widetilde{\mu}_S(\sg)
=\int_{\mathcal F}\int_{\mathcal T}
\sum_{\vk\in \ZZ_S^d} f_\sg(\vk+\vw) d\vw d\widetilde{\mu}_S(\sg)\\
&=\int_{\mathcal F}\int_{\QQ_S^d} f_\sg(\vw) d\vw d\widetilde{\mu}_S(\sg)
=\int_{\mathcal F}\int_{\QQ_S^d} f \:d\vw d\widetilde{\mu}_S(\sg)
=\int_{\QQ_S^d} f \:d\vw.
\end{split}\]

\vspace{0.1in}
\noindent (b) Similarly, by the change of variables,
\[\begin{split}
&\int_{\mathcal F}\int_{\mathcal T\sg} \sum_{\vk_1, \vk_2\in \ZZ_S^d}
F(\vk_1\sg+\vv, \vk_2\sg+\vv) \:d\vv d\widetilde{\mu}_S(\sg)\\
&\hspace{1in}=\int_{\mathcal F}\int_{\mathcal T} \sum_{\vk_1, \vk_2\in \ZZ_S^d}
F\left((\vk_1+\vw)\sg, (\vk_2+\vw)\sg\right)\:d\vw d\widetilde{\mu}_S(\sg).
\end{split}\]
Let $\vw'=\vk_1+\vw$ and $\vk_3=\vk_2-\vk_1$. Then the above integral is
\[\begin{split}
&=\int_{\mathcal F}\left(\sum_{\vk_1\in \ZZ_S^d}\int_{\mathcal T+\vk_1}\right)
\sum_{\vk_3\in \ZZ_S^d} F(\vw'\sg, \vw'\sg+ \vk_3\sg) \:d\vw' d\widetilde{\mu}_S(\sg)\\
&=\int_{\mathcal F} \sum_{\vk_3\in \ZZ_S^d} \int_{\QQ_S^d}
F(\vx, \vx+\vk_3\sg) \:d\vx d\widetilde{\mu}_S(\sg)\\
&=\int_{\QQ_S^d} F(\vx, \vx) d\vx 
+ \int_{\QQ_S^d}\int_{\QQ_S^d} F(\vx, \vy)\: d\vx d\vy,
\end{split}\]
where the last equality follows from the Siegel integral formula for $\UL_d(\ZZ_S)\setminus \UL_d(\QQ_S)$ (\cite[Proposition 3.11]{HLM2017}) with the function $\vy \in \QQ_S^d \mapsto \int_{\QQ_S^d} F(\vx, \vy) d\vx$.
\end{proof}

We remark that the proof of Theorem~\ref{moment formulae: inhomogeneous} (b) is a generalization of the proof of Proposition 14 in \cite{EMV2015}, where they considered the case when $\QQ_S=\RR$ and $d=2$.

By Theorem~\ref{moment formulae: inhomogeneous} and Chebyshev's inequality, we obtain the theorem below immediately.

\begin{theorem}\label{variance property: inhomogeneous}
Let $d\ge 2$. Let $\mathcal F_{\mathrm{aff}}$ be a fundamental domain for $\AUL_d(\ZZ_S)\setminus \AUL_d(\QQ_S)$.
For a bounded borel set $A\subseteq \QQ_S^d$, 
\[\begin{split}
&\widetilde{\mu}_S\times\vol\left(\left\{(\vv,\sg) \in \mathcal F_{\mathrm{aff}}
: \left| \#\left((\ZZ_S^d\sg+\vv) \cap A\right) -\vol(A) \right| > M
\right\}\right)\\
&\hspace{0.4in} < \frac {\vol(A)} {M^2}.
\end{split}\]
\end{theorem}

\begin{proof}[Proof of Theorem~\ref{inhomogeneous case}]
Let $\mathcal F=\mathcal F_{\infty}\times \prod_{p\in S_f} \UL_d(\ZZ_p)$ be a fundamental domain of $\UL_d(\ZZ_S)\setminus \UL_d(\QQ_S)$, where $\mathcal F_{\infty}$ is a fundamental domain for $\SL_d(\ZZ)\setminus \SL_d(\RR)$.
Take a fundamental domain $\mathcal F_{\mathrm{aff}}$ of $\AUL_d(\ZZ_S)\setminus \AUL_d(\QQ_S)$ by $\mathcal F_{\mathrm{aff}}=\bigcup_{\sg\in \mathcal F} \mathcal T\sg \times\{\sg\}$, where $\mathcal T=[0,1)^d\times \prod_{p\in S_f} \ZZ_p^d$. 

Let $\mathcal K$ be a compact subset of $\mathcal F_{\mathrm{aff}}$. For $\J\in \NN \times \prod_{p\in S_f} p^{\NN}$, define
\[
\mathcal C_{\J}
=\left\{
(\xi, \sg)\in \mathcal K: \begin{array}{c}
\hspace{-0.4in}D\left(\ZZ_S^d, A_{(\q{\sg}{0})_\xi,\T}\right)>\vol(\I_\T)
\prod_{p\in S} T_p^{d-2-\delta_p}\\[0.05in]
\hspace{1.2in}\text{for some }\T \in [j_\infty, j_\infty+1)\times \J^f\end{array}
\right\}.
\]

As in the proof of Theorem~\ref{congruence case}, we will show that $\vol\times\widetilde{\mu}_S\left(\limsup_{\J}\mathcal C_\J\right)=0$ for an appropriate $(\delta_p)_{p\in S_f}$.

For $(\eta, \sh)$, $(\xi, \sg)$ in $\mathcal F_{\mathrm{aff}}$, $\I\in \QQ_S$ and $\T=(T_p)_{p\in S}$,
\[\begin{split}
&\vv \in \ZZ_S^d: \;
\q{\sh\sg}{0}:\; \q{(\eta,\sh)(\xi,\sg)}{0}(\vv)\in \I, \;\|\vv\|_p<T_p\;(p\in S)\\
&\Rightarrow \vw\in \ZZ_S^d\sg+\xi:\;
\q{(\eta,\sh)}{0}:\; \q{(\eta,\sh)}{0}(\vw)\in \I, \;\|\vw\|_p<\|(\xi_p,g_p)\|_{op}T_p\;(p\in S).
\end{split}\]
Note that for $\sg=(g_p)_{p\in S}\in \mathcal F$, $\ZZ_p^d g_p=\ZZ_p^d$ ($p\in S_f$) so that for $(\xi, \sg)\in \mathcal F_{\mathrm{aff}}$,
\[
\|\vw\|_p=\|\vv g_p + \xi_p\|_p \le \max\left\{\|\vv\|_p, \|\xi_p\|_p \right\}
\]
and if $\|\vv\|_p=T_p > 1$, $\|\vv\|_p=\|\vw\|_p$ for $p\in S_f$.

For each $\J$, as in the proof of Theorem~\ref{congruence case}, 
set $\varepsilon_1=j_\infty^{-\alpha_\infty}\prod_{p\in S_f}p^{-\alpha_p j_p}$ and $\varepsilon_2=j_\infty^{-\beta_\infty}\prod_{p\in S_f}^{-\beta_p j_p}$, which are the scales of the space-approximation and the radius-approximation, respectively.

Observe that there is $C=C(\mathcal K)>0$ such that for any $\J$, one can find a subset $\mathcal Q=\mathcal Q(\mathcal K, \J)$ of $\mathcal K$ for which
\begin{enumerate}[(i)]
\item $\mathcal K \subseteq \bigcup_{(\eta,\sh)\in \mathcal Q} (\eta,\sh).\mathcal B(\varepsilon_1)$, where
\[\begin{split}
&\mathcal B(\varepsilon_1)=\\
&\left\{
(\xi_\infty, g_\infty) \in \RR^d\rtimes \SL_d(\RR) : \|(\xi_\infty, g_\infty)\|_{op}\le 1+\varepsilon_1
\right\}\times \prod_{p\in S_f} (\ZZ_p^d \rtimes \UL_d(\ZZ_p));
\end{split}\]
\item $\# \mathcal Q(\mathcal K, \J) < C(\mathcal K)\varepsilon^{-\frac 1 2 (d+2)(d-1)-d},$
\end{enumerate}
where $\frac 1 2 (d+2)(d-1)+d$ is the codimension of $\{\origin\}\times \SO(d)$ in $\SL_d(\RR)$.

Then one can show that $\sum_{\J\succ \J_0} \vol\times \widetilde{\mu}_S(\mathcal C_\J)<\infty$ for sufficiently large $\J_0$ provided that
\[\begin{split}
&\left\{\begin{array}{l}
0<\delta_\infty < \alpha_\infty,\\
\delta_\infty < \beta_\infty +1,\\
-(d-2-\kappa_\infty)-2\delta_\infty+\alpha_\infty\left(\frac 1 2 (d+2)(d-1)+d\right)+\beta_\infty < -1;
\end{array}
\right.\\[0.1in]
&\left\{\begin{array}{l}
0<\delta_p < \alpha_p,\\
\delta_p < \beta_p,\\
-(d-2-\kappa_p)-2\delta_p+\alpha_p\left(\frac 1 2 (d+2)(d-1)+d\right)+\beta_p < 0
\end{array}
\right.
\end{split}\]
and the theorem follows from Borel-Cantelli lemma.
Note that the range of such $\delta_p$ are $\left(0, \frac {d-2-\kappa_p} {(d+2)(d-1)/2+d-1}\right)$. 

\end{proof}

\end{document}